\documentclass[a4paper,12pt]{amsart}
\usepackage{amsfonts}
\usepackage{amsmath,array}
\usepackage{amssymb}
\usepackage{mathrsfs}
\usepackage[colorlinks]{hyperref}
 \usepackage{graphicx}
\usepackage{enumerate}
\usepackage[usenames]{color}
\newtheorem{thm}{Theorem}[section]
\newtheorem{cor}[thm]{Corollary}
\newtheorem{lem}[thm]{Lemma}
\newtheorem{prop}[thm]{Proposition}

\numberwithin{equation}{section}
\usepackage[english]{babel} 
\usepackage{blindtext}

\theoremstyle{definition}
\newtheorem{definition}[thm]{Definition}
\newtheorem{rem}[thm]{Remark}

\begin{document}
 \title[Sobolev inequality on noncommutative Euclidean spaces]{Sobolev inequality and its applications to nonlinear PDE on noncommutative Euclidean spaces}

\author[Michael Ruzhansky]{Michael Ruzhansky}
\address{
 Michael Ruzhansky:
  \endgraf
 Department of Mathematics: Analysis, Logic and Discrete Mathematics,
  \endgraf
 Ghent University, Ghent,
 \endgraf
  Belgium 
  \endgraf
  and 
 \endgraf
 School of Mathematical Sciences, Queen Mary University of London, London,
 \endgraf
 UK  
 \endgraf
  {\it E-mail address} {\rm michael.ruzhansky@ugent.be}
  }

\author[Serikbol Shaimardan]{Serikbol Shaimardan}
\address{
  Serikbol Shaimardan:
  \endgraf
  Institute of Mathematics and Mathematical Modeling, 050010, Almaty, 
  \endgraf
  Kazakhstan 
  \endgraf
  and 
  \endgraf
Department of Mathematics: Analysis, Logic and Discrete Mathematics
  \endgraf
 Ghent University, Ghent,
 \endgraf
  Belgium
  \endgraf
  {\it E-mail address} {\rm shaimardan.serik@gmail.com} 
  }

 \author[Kanat Tulenov]{Kanat Tulenov}
\address{
  Kanat Tulenov:
  \endgraf
  Institute of Mathematics and Mathematical Modeling, 050010, Almaty, 
  \endgraf
  Kazakhstan 
  \endgraf
  and 
  \endgraf
Department of Mathematics: Analysis, Logic and Discrete Mathematics
  \endgraf
 Ghent University, Ghent,
 \endgraf
  Belgium
  \endgraf
  {\it E-mail address} {\rm kanat.tulenov@ugent.be} 
  }

\date{}

\begin{abstract}
In this work, we study the Sobolev inequality on noncommutative Euclidean spaces. As a simple consequence, we obtain the Gagliardo–Nirenberg type inequality and as its application we show global well-posedness of nonlinear PDEs in the noncommutative Euclidean space. Moreover, we show that the logarithmic Sobolev inequality is equivalent to the Nash inequality for possibly different constants in this noncommutative setting by completing the list in noncommutative Varopoulos's theorem in \cite{Zhao}.  Finally, we present a direct application of the Nash inequality to compute the time decay for solutions of the heat equation in the noncommutative setting. 
\end{abstract}



\maketitle

\tableofcontents
{\section{Introduction}}

The classical Sobolev inequality \cite{Sobolev} in the Euclidean space $\mathbb{R}^d,$  $d>2,$ states that, for any sufficiently smooth function $f$ with compact support, the inequality 
\begin{equation}\label{main-sobolev-ineq}\|f\|_{2d/d-2}\leq C\|\nabla f\|_{2},
\end{equation}
  holds, where the constant $C>0$ depends only on $d>2.$ This inequality is widely applied in the study of partial differential equations (briefly PDEs) (e.g., \cite{GT}) and is interconnected with various other inequalities. The Sobolev inequality finds widespread applications in diverse areas of mathematics and physics. In the study of elliptic and parabolic partial differential equations, it serves as a fundamental tool for proving existence and uniqueness of solutions, as well as regularity properties of solutions. Moreover, in the context of variational problems and optimization, the Sobolev inequality plays a crucial role in establishing the compactness of minimizing sequences and the convergence of solutions. Extending this theory to noncommutative settings poses unique challenges and opportunities. In the noncommutative context, the traditional notion of differentiation must be reinterpreted, often relying on concepts from operator theory and functional analysis. In the framework of the non-commutative geometry, the notion of Euclidean space undergoes a profound transformation. Instead of traditional coordinates that commute with each other, we deal with noncommuting coordinates, reflecting the noncommutative nature of the space. This departure from commutativity introduces intriguing mathematical structures, reminiscent of quantum mechanics and operator algebras. Noncommutative Euclidean space $\mathbb{R}^d_{\theta},$ which is defined in terms of an arbitrary skew-symmetric real $d\times d$ matrix $\theta,$ represent deformations of Euclidean space $\mathbb{R}^d$
  as outlined by Rieffel \cite{Rieffel}. This family of spaces stands among the earliest and most thoroughly investigated examples in noncommutative geometry, owing to its significance in the quantum mechanical phase space perspective. Various scholars, including Moyal \cite{M} and Groenwald \cite{Gro}, have approached these spaces from diverse angles. The central concept involves deforming the algebra of smooth functions on $\mathbb{R}^d$ by substituting the standard pointwise product with the twisted Moyal product. In the realm of noncommutative geometry, noncommutative Euclidean spaces represent notable instances of ``noncompact"  spaces \cite{CGRS, GGSVM}. The Moyal product garners attention for its relevance to quantum phase space \cite[Chapter 13]{H}, \cite{green-book}. Furthermore, noncommutative Euclidean spaces play a crucial role in quantum mechanics, especially in situations where spatial coordinates fail to commute \cite{CGRS}. The operators defined on $\mathbb{R}^d_{\theta}$, referred to as quantum-classical Weyl operators, have been thoroughly investigated in \cite{BW, DW}. The concept of convolution on $\mathbb{R}^d_{\theta}$, under the assumption that $\theta$ is a real, invertible, skew-symmetric matrix, was first introduced in \cite{W}. That work also proved the associativity of the convolution and emphasized the role of the Fourier-Weyl transform, which is treated in the present context as the Fourier transform on $\mathbb{R}^d_{\theta}$. The more general case involving arbitrary $\theta$ is addressed in \cite{DW}.
In the literature, multiple equivalent constructions regarding noncommutative Euclidean space are documented. One approach to defining $\mathbb{R}^d_{\theta}$ starts by establishing the von Neumann algebra $L^{\infty}(\mathbb{R}^{d}_{\theta})$ as a twisted left-regular representation of $\mathbb{R}^d$ on $L^2(\mathbb{R}^d).$
We will delve into this definition of a von Neumann algebra, which is generated by a $d$-parameter strongly continuous unitary family $\{U_{\theta}(t)\}_{t\in\mathbb{R}^d}$ satisfying the relation
$$
U_{\theta}(t)U_{\theta}(s)=e^{\frac{1}{2}i(t,\theta s)}U_{\theta}(t+s),\quad t,s\in \mathbb{R}^d,
$$
where $(\cdot,\cdot)$ is inner product in $\mathbb{R}^d.$

In recent times, there has been a significant body of research aimed at extending the techniques of the classical harmonic analysis on Euclidean spaces to noncommutative setting. This is because it is possible to define analogues of many tools of harmonic analysis, such as differential operators, Fourier transform, and function spaces. Recent progress in this theory even allow to study nonlinear PDEs \cite{Mc}. For more details and recent results on this theory, we refer the reader to \cite{GJM}, \cite{GGSVM}, \cite{MSX},  \cite{Mc}, \cite{HLW}, \cite{RST}, and references therein.

There are several proofs of the Sobolev inequality in $\mathbb{R}^d,$ $d>2.$ One of them is by using the heat kernel estimate \cite{VSC}. Moreover, N. Varopoulos obtained an equivalence between a heat kernel estimate and the Sobolev inequality in an abstract settings in 1985 \cite{Var}. He proved that the Sobolev inequality is equivalent to the heat kernel estimate as well as the Nash \cite{Nash}, and Moser \cite{Moser} inequalities, for possibly different constants. On the other hand, the Nash inequality is also equivalent to the logarithmic Sobolev inequality. The Sobolev inequality was proved in \cite[Theorem 3.11 and Corollary 3.19]{Zhao} by using the heat kernel estimate. Moreover, the author obtained the noncommutative analogue of the Varopoulos's theorem even a more general case which includes noncommutative Euclidean spaces. On the other hand, there is another way to prove the Sobolev inequality which hinges upon the Young inequality for weak type spaces and the Hardy-Littlwood-Sobolev inequality. {\color{red}These inequalities including the Sobolev inequality in the context of quantum phase space were explored recently by Lafleche in \cite{L}. The quantum phase space corresponds, up to unitary equivalence, to a special case of the noncommutative Euclidean space when $\theta$ is an invertible skew-symmetric matrix. A central distinction in our work lies in addressing the general situation where $\theta$ may not be invertible. In the invertible case, the $L^p$-spaces over $\mathbb{R}^d_{\theta}$ are isomorphic to Schatten classes, and there exist well-defined embeddings between $L^p$-spaces for varying values of $p$. In contrast, such inclusions fail in the general setting, introducing substantial analytical differences. This distinction is essential for the inequalities and results we establish for partial differential equations. }
Therefore, we first focused to obtain the Hardy-Littlwood-Sobolev and Young inequality for weak type spaces in the noncommutative Euclidean space. Our proof is based on the Young inequality in \cite{Mc} and the interpolation in the classical Lorentz spaces as well as a proper extension of the definition of convolutions in the class of distributions.  
Moreover, we prove that the equivalence between the Nash inequality and the logarithmic Sobolev inequality remains true even in the noncommutative Euclidean space by completing the list of the noncommutative Varopoulos's theorem in \cite[Theorem 4.30]{Zhao}. Note that the Logarithmic Sobolev inequalities were originally introduced by Gross \cite{Gross} as a reformulation of hypercontractivity, and have since been the subject of extensive research. As a consequence of the noncommutative Sobolev and  H\"{o}lder inequalities we obtain a version of the Gagliardo–Nirenberg  inequality which allows us to show global well-posedness of nonlinear damped wave equations for the sub-Laplacian in the noncommutative Euclidean space. Nonlinear PDEs in the noncommutative Euclidean space were studied very recently in \cite{Mc}, where the author obtained local (some global) well-posedness in time of Allen-Cahn, Schr\"{o}dinger and incompressible Navier-Stokes equations. Additionally, a recent contribution by the authors in \cite{RST2} investigates $L^p$-$L^q$ norm estimates for solutions of heat, wave, and Schr\"odinger type equations with Caputo fractional derivative and well-posedness of nonlinear heat and wave equations in this noncommutative framework. {\color{red}The main difference in this work is that we study damped wave equations for the sub-Laplacian and show its global in time well posedness.  In the commutative case when $\theta=0$, the study of linear and nonlinear damped wave equations has a rich historical background. The authors in \cite{von, Matsumura} were the first to consider this type of problem for the Laplacian on $\mathbb{R}^d$. Further developments can be found in \cite{HKN04, HKN06, HL92, HO04, Ike04,KU13, Kha13, RT19}, which addresses damped wave equations in  $\mathbb{R}^d$ under various assumptions. These works, along with the references therein, focus on the global solvability of Cauchy problems for nonlinear wave equations involving the Laplace operator and a dissipative term.  Moreover, our approach is completely different from that of \cite{Mc}. However, our approach is also applicable at least to obtain a global in time well posedness of the nonlinear Schr\"{o}dinger equation which was studied in \cite{Mc}. In general,  there are some technical difficulties in noncommutative setting including defining even nonlinear operator functions. The other difficulties in the noncommutative Euclidean spaces were explained in \cite{Mc}.
{\color{red} The central objective in this setting is to derive a Gagliardo-Nirenberg type inequality as a consequence of the Sobolev inequality and to employ an appropriate version of the Banach fixed point theorem to establish the desired results.} The idea comes from the paper \cite{RT}, where the authors studied similar problems in the context of Heisenberg and graded Lie groups.}
At the end, as in the classical case, we show a direct application of the Nash inequality to compute the time decay for solutions of the heat equation in the noncommutative setting.


\section{Preliminaries}
 
 \subsection{Noncommutative (NC) Euclidean space $\mathbb{R}^{d}_{\theta}$} \label{NC Euclidean space}

For a thorough examination of the noncommutative (NC) Euclidean space $\mathbb{R}^{d}_{\theta}$ and additional insights, we recommend recent works \cite{GJM}, \cite{GGSVM}, \cite{HLW}, \cite{MSX}, \cite{Mc}, and \cite{RST}.

Let $H$ represents a Hilbert space, with $B(H)$ denoting the algebra consisting of all bounded linear operators that act on $H$. In the usual context, we denote by $L^p(\mathbb{R}^d)$ ($1\leq p<\infty$) the $L^p$-spaces of pointwise almost-everywhere equivalence classes of $p$-integrable functions, while $L^{\infty}(\mathbb{R}^d)$ denotes the space of essentially bounded functions on the Euclidean space $\mathbb{R}^d$.

For $1\leq p\leq\infty$, the Lorentz space $L^{p,\infty}(\mathbb{R}^d)$ refers to the space of complex-valued measurable functions $f$ on $\mathbb{R}^d$ satisfying the finiteness of the quasinorm defined as:
\begin{equation}\label{weak-Lp-commut}
\|f\|_{L^{p,\infty}(\mathbb{R}^d)}=\sup\limits_{t>0}t^\frac{1}{p}f^*(t)=\sup\limits_{s>0}sd(s,f)^\frac{1}{p}, 
\end{equation}
where $d(\cdot,f)$ represents the distribution of the function $f$. For a more comprehensive understanding of these spaces, we refer the reader to \cite{G2008}.

Let us suppose we have an integer $d\geq 1$, and we choose an antisymmetric $\mathbb{R}$-valued $d\times d$ matrix $\theta=\{\theta_{j,k}\}_{1\leq j,k\leq d}.$
\begin{definition}\label{NC_E_space1}Define $\mathbb{R}^{d}_{\theta}$ (or $L^{\infty}(\mathbb{R}^{d}_{\theta}))$ as the von Neumann algebra generated by the $d$-parameter strongly
continuous unitary family $\{U_{\theta}(t)\}_{t\in \mathbb{R}^{d}}$ satisfying the relations
\begin{equation}\label{weyl-relation}
U_{\theta}(t)U_{\theta}(s)=e^{\frac{1}{2}i(t,\theta s)}U_{\theta}(t+s),\quad t,s\in \mathbb{R}^d,
\end{equation}
where $(\cdot,\cdot)$ means the usual inner product in $\mathbb{R}^d.$
\end{definition}
The relation mentioned above is known as the Weyl representation of the canonical commutation relation. While it is viable to define $L^{\infty}(\mathbb{R}_{\theta}^d)$ in an abstract operator-theoretic manner as outlined in \cite{GGSVM}, an alternative approach involves defining the algebra through a specific set of operators defined on the Hilbert space $L^2(\mathbb{R}^d)$.  
\begin{definition} \cite[Definition 2.1]{Mc}\label{NC_E_space2} Let $U_{\theta}(t)$ denote the operator on $L^{2}(\mathbb{R}^{d})$ for $t\in\mathbb{R}^d$, defined by:
    $$(U_{\theta}(t)\xi)(s)=e^{i(t,s)}\xi(s-\frac{1}{2}\theta t),\quad \xi\in L^{2}(\mathbb{R}^{d}),  \, t,s\in \mathbb{R}^d.$$
It can be demonstrated that the family $\{U_{\theta}(t)\}_{t\in \mathbb{R}^{d}}$ is strongly continuous and satisfies the relation \eqref{weyl-relation}.
Subsequently, the von Neumann algebra $L^{\infty}(\mathbb{R}_{\theta}^d)$ is defined as the weak operator topology closed subalgebra of $B(L^{2}(\mathbb{R}^{d}))$ generated by the family $\{U_{\theta}(t)\}_{t\in \mathbb{R}^{d}}$, and is called a {\color{red}noncommutative} (or quantum) Euclidean space.
\end{definition}
{\color{red} It should be noted that when $\theta=0$,} the definition provided above reduces to characterizing $L^{\infty}(\mathbb{R}^{d})$ as the algebra of bounded pointwise multipliers on $L^{2}(\mathbb{R}^{d})$.

Generally, the algebraic nature of the noncommutative Euclidean space $L^{\infty}(\mathbb{R}_{\theta}^d)$ depends on the dimension of the kernel of $\theta.$ In the case $d=2,$ up to an orthogonal conjugation $\theta$ may be given as 
\begin{equation}\label{d=2}
    \theta=h\begin{pmatrix}
0 & -1\\
1 & 0
\end{pmatrix} 
\end{equation}
for some constant $h>0.$ In this case, $L^{\infty}(\mathbb{R}_{\theta}^2)$ is $*$-isomorphic to $B(L^{2}(\mathbb{R}))$ and this $*$-isomorphism can be written as 
$$U_{\theta}(t)\to e^{it_1\mathbf{M}_s +it_2 h \frac{d}{ds}},$$
where $\mathbf{M}_s\xi(s)=s\xi(s)$ and $\frac{d}{ds}\xi(s)=\xi'(s)$ is the differentiation.
If $d\geq 2,$ then an arbitrary $d\times d$ antisymmetric real matrix can be expressed (up to orthogonal conjugation) as a direct sum of a zero matrix and matrices of the form \eqref{d=2}, ultimately leading to the $*$-isomorphism
\begin{equation}\label{direct-sum}
    L^{\infty}(\mathbb{R}_{\theta}^d)\cong L^{\infty}(\mathbb{R}^{\dim(\ker(\theta))})\bar{\otimes}B(L^{2}(\mathbb{R}^{\text{rank}(\theta)/2})),
\end{equation}
where $\bar{\otimes}$ is the von Neumann tensor product \cite{MSX}. In particular, if $\det(\theta)\neq 0,$ then \eqref{direct-sum} reduces to
\begin{equation}\label{reduced-direct-sum}
    L^{\infty}(\mathbb{R}_{\theta}^d)\cong B(L^{2}(\mathbb{R}^{d/2})).
\end{equation}
It is worth noting that these formulas make sense due to the fact that the rank of an antisymmetric matrix is always even.
\subsection{Noncommutative integration} 
Given $f\in L^{1}(\mathbb{R}^d),$ we define the  operator $\lambda_{\theta}(f)$   by the formula
\begin{equation}\label{def-integration}
\lambda_{\theta}(f)\xi=\int_{\mathbb{R}^d}f(t) U_{\theta}(t)\xi dt, \quad \xi\in L^{2}(\mathbb{R}_{\theta}^d).
\end{equation}
This integral converges absolutely in the Bochner sense, yielding a bounded linear operator $\lambda_{\theta}(f): L^{2}(\mathbb{R}^d)\to L^{2}(\mathbb{R}^d)$ such that $\lambda_{\theta}(f)\in L^{\infty}(\mathbb{R}_{\theta}^d)$ (see \cite[Lemma 2.3]{MSX}). 
Let us use $\mathcal{S}(\mathbb{R}^d)$ to represent the classical Schwartz space on $\mathbb{R}^d.$ For any $f\in\mathcal{S}(\mathbb{R}^d)$ we define the Fourier transform as:
$$
\widehat{f}(t)=\int_{\mathbb{R}^d}f(s)e^{-i(t,s)}ds, \quad t\in \mathbb{R}^d.
$$ 
The noncommutative Schwartz space $\mathcal{S}(\mathbb{R}_{\theta}^d)$ is precisely the set of elements of $L^{\infty}(\mathbb{R}_{\theta}^d)$ that can be expressed as  $\lambda_{\theta}(f)$ for some  $f$ belonging to the classical Schwartz space $\mathcal{S}(\mathbb{R}^d).$ In other words:
$$
\mathcal{S}(\mathbb{R}_{\theta}^d):=\{x\in L^{\infty}(\mathbb{R}_{\theta}^d):x=\lambda_{\theta}(f) \,\ \text{for some}\,\ f\in \mathcal{S}(\mathbb{R}^d)\}.
$$
We equip $\mathcal{S}(\mathbb{R}_{\theta}^d)$
with a topology induced by the canonical Fr\'{e}chet topology on   $\mathcal{S}(\mathbb{R}^d)$ via the map $\lambda_{\theta}.$ 
 The topological dual of $\mathcal{S}(\mathbb{R}_{\theta}^d)$ is denoted as  $\mathcal{S}'(\mathbb{R}_{\theta}^d).$ 
 Moreover, the injectivity of the mapping $\lambda_{\theta}$ is established \cite[Subsection 2.2.3]{MSX}, \cite{Mc}. {\color{red}For any $g \in \mathcal{S}(\mathbb{R}^d),$ the reflection   $\widetilde{g}$ is given by $\widetilde{g}(\xi) := g(-\xi),$ $\xi \in \mathbb{R}^d,$ and this convention is used throughout the paper.} Therefore, the mapping $\lambda_{\theta}$ can be extended to distributions 
 \begin{eqnarray}\label{injectivity}
(\lambda_{\theta}(f), \lambda_{\theta}(g))=\langle f, \widetilde{g}\rangle \quad \text{for all} \quad f,g\in\mathcal{S}(\mathbb{R}^d). 
\end{eqnarray}
For any $f\in \mathcal{S}(\mathbb{R}^d)$, we define the functional $\tau_{\theta}:\mathcal{S}(\mathbb{R}_{\theta}^d)\to \mathbb{C}$  with the formula:
\begin{equation}\label{trace-def}
\tau_{\theta}(\lambda_{\theta}(f))=\tau_{\theta}\left(\int_{\mathbb{R}^d}f(\eta)U_{\theta}(\eta)d\eta\right):= f(0)
\end{equation}

This functional $\tau_{\theta}$ can be extended to a semifinite normal trace on $L^{\infty}(\mathbb{R}_{\theta}^d)$. Furthermore, if $\theta=0$, then $\tau_{\theta}$ coincides exactly with the Lebesgue integral under a suitable isomorphism. If $\det(\theta)\neq0$, then $\tau_{\theta}$ is (up to normalization) the operator trace on $B(L^2(\mathbb{R}^{d/2}))$. For further details, we refer to \cite{GJP}, \cite[Lemma 2.7]{MSX}, \cite[Theorem 2.6]{Mc}.
\subsection{Noncommutative $L^{p}(\mathbb{R}^{d}_{\theta})$ and $L^{p,\infty}(\mathbb{R}^{d}_{\theta})$ spaces}
Given the definitions outlined in the earlier sections, $L^{\infty}(\mathbb{R}^{d}_{\theta})$ emerges as a semifinite von Neumann algebra, with   $\tau_{\theta}$ serving as its trace. Thus, the pair $(L^{\infty}(\mathbb{R}^{d}_{\theta}),\tau_{\theta})$ 
 is characterized as a noncommutative measure space. We can define the  $L^p$-norm on this space for any $1\leq p<\infty$  using the Borel functional calculus and the following expression:
$$\|x\|_{L^p(\mathbb{R}^d_{\theta})}=\Big(\tau_{\theta}(|x|^p)\Big)^{1/p},\quad x\in L^{\infty}(\mathbb{R}^{d}_{\theta}), $$
where $|x|:=(x^{*}x)^{1/2}.$ The space $L^p(\mathbb{R}^d_{\theta})$  is the completion of the set $\{x\in L^{\infty}(\mathbb{R}^{d}_{\theta}) :\|x\|_{p}<\infty\}$ under the norm $\|\cdot\|_{L^p(\mathbb{R}^d_{\theta})},$ and this completion is denoted by $L^p(\mathbb{R}^d_{\theta}).$ The elements of $L^p(\mathbb{R}^d_{\theta})$ are $\tau_{\theta}$-measurable operators, like in the commutative case. These are linear densely defined closed (possibly unbounded) affiliated with $L^{\infty}(\mathbb{R}^{d}_{\theta})$ operators such that $\tau_{\theta}(\mathbf{1}_{(s,\infty)}(|x|))<\infty$ for some $s>0.$ Here, $\mathbf{1}_{(s,\infty)}(|x|)$ denotes the spectral projection corresponding to the interval $(s,\infty).$ 
Let $L^{0}(\mathbb{R}^{d}_{\theta})$ 
 represent the collection of all
$\tau_{\theta}$-measurable operators.  
 The {\it distribution function} of $x$ is defined by
$$
d(s;x)=\tau_{\theta}\left(\mathbf{1}_{(s,\infty)}(|x|)\right), \quad 0<s<\infty.
$$
For $x\in L^{0}(\mathbb{R}^{d}_{\theta}),$ we define the {\it generalised singular value function} $\mu(t, x)$ as follows:
\begin{equation}\label{distribution-function}
\mu(t,x)=\inf\left\{s>0: d(s;x)\leq t\right\}, \quad t>0.
\end{equation}
The function $t\mapsto\mu(t,x)$ is decreasing and right-continuous. For further discussion on generalized singular value functions, we direct the reader to \cite{FK, LSZ}. The norm of $L^p(\mathbb{R}^d_{\theta})$ can also be represented using the generalized singular value function, as outlined in (see \cite[Example 2.4.2, p. 53]{LSZ}): 
\begin{equation}\label{mu-norm}
\|x\|_{L^p(\mathbb{R}^d_{\theta})}=\left(\int_{0}^{\infty}\mu^{p}(s,x)ds\right)^{1/p}, \,\ \text{if} \,\ p<\infty, \|x\|_{L^{\infty}(\mathbb{R}^d_{\theta})}=\mu(0,x), \,\, \text{if}\,\ p=\infty.
\end{equation}
The equality for $p=\infty,$ was established in  \cite[Lemma 2.3.12. (b), p. 50]{LSZ}.

The space $L^{0}(\mathbb{R}^{d}_{\theta})$ is a $*$-algebra, which can be endowed with a topological  $*$-algebra structure in the following manner.  We consider the set
$$V(\varepsilon,\delta)=\{x\in L^{0}(\mathbb{R}^{d}_{\theta}): \mu(\varepsilon,x)\leq \delta\}.$$
Then, the family $\{V(\varepsilon,\delta): \varepsilon,\delta>0\}$ constitutes a system of neighbourhoods at 0, resulting in $L^{0}(\mathbb{R}^{d}_{\theta})$  becoming a metrizable topological $*$-algebra. The convergence induced by this topology is called the {\it convergence in measure} \cite{PXu}. Next, we establish the noncommutative Lorentz space linked with the noncommutative Euclidean space.

For $1\leq p\leq\infty,$ we introduce the non-commutative Lorentz space $L^{p,\infty}(\mathbb{R}^{d}_{\theta})$ as follows:
$$\|x\|_{L^{p,\infty}(\mathbb{R}^{d}_{\theta})}:=\sup_{t>0}t^{\frac{1}{p}}\mu(t,x)=\sup_{s>0}sd(s;x)^{\frac{1}{p}}.
$$

These spaces constitute noncommutative quasi-Banach spaces. For an extensive exploration of  $L^p$ and Lorentz spaces corresponding to general semifinite von Neumann algebras, we suggest referring to \cite{DPS}, \cite{LSZ}, \cite{PXu}.

\subsection{Differential calculus on $L^{\infty}(\mathbb{R}^d_\theta)$} 
Let us recall the differential structure on $L^{\infty}(\mathbb{R}^d_\theta),$ as detailed in (see, \cite[Subsection 2, p. 10]{Mc}).

The differential structure relies on the group of translations $\{T_\eta\}_{\eta\in\mathbb{R}^d},$ where $T_\eta$ is presented as the unique $\ast$-automorphism of $L^{\infty}(\mathbb{R}^d_\theta)$ that operates on $U_{\theta}(\xi)$ as follows:
\begin{equation}\label{def-translations-1}
T_\eta(U_{\theta}(\xi))=e^{i(\eta,\xi)}U_{\theta}(\xi), \quad  
\xi,\eta \in\mathbb{R}^d, 
\end{equation}
where $(\cdot,\cdot)$ denotes the standard inner product in $\mathbb{R}^d.$

Expressed in terms of the map  $\lambda_{\theta},$ we observe:
\begin{eqnarray}\label{property-translation}
T_\eta(\lambda_{\theta}(f))=\lambda_{\theta}(e^{i(\eta,\cdot)}f(\cdot)),    
\end{eqnarray}
for all $f\in \mathcal{S}(\mathbb{R}^d).$

Alternatively, for $x \in L^{\infty}(\mathbb{R}^d_\theta) \subseteq B(L^2(\mathbb{R}^d_\theta)),$  we can introduce  $T_\eta(x)$ as the conjugation of $x$ by the unitary operator representing translation by  $\eta$  on $L^2(\mathbb{R}^d_\theta).$   
\begin{definition} (see, \cite[Definition 2.9]{Mc}) An element $x \in L^{1}(\mathbb{R}^d_\theta) + L^{\infty}(\mathbb{R}^d_\theta)$   is considered smooth if, for all  
$y \in L^{1}(\mathbb{R}^d_\theta) \cap L^{\infty}(\mathbb{R}^d_\theta)$ the function $\eta \mapsto \tau_{\theta}(yT_\eta(x))$ is smooth.
\end{definition}

The partial derivatives  $\partial^{\theta}_j,$ where $  j = 1, \dots, d,$ are defined on smooth elements $x$ as follows:  
$$
\partial^{\theta}_j(x) = \frac{d}{ds_{j}}T_\eta(x)|_{\eta=0}. 
$$
Using  \eqref{def-integration} and \eqref{def-translations-1}, we can readily confirm that for $x=\lambda_{\theta}(f),$  the following holds:
\begin{equation}\label{operator-derivation}
\partial^{\theta}_j(x)\overset{\eqref{def-integration}}{=}\partial^{\theta}_j\lambda_{\theta}(f)\overset{\eqref{def-translations-1}}{=}\lambda_{\theta}(it_{j}f(t)), \quad  j=1,\cdots,d, \,\  f\in \mathcal{S}(\mathbb{R}^d).
\end{equation}
We introduce the notation  $\partial^{\alpha}_{\theta}$
for a multi-index $\alpha=(\alpha_1,...,\alpha_d)$ as follows:
$$
\partial^{\alpha}_{\theta}=(\partial^{\theta}_{1})^{\alpha_{1}}\dots(\partial^{\theta}_{d})^{\alpha_{d}}. 
$$
{\color{red}The operator $\nabla_{\theta}$ is defined by
$$
\nabla_{\theta} = \left( \partial^{\theta}_1, \dots, \partial^{\theta}_d \right),
$$
with domain 
$$
\bigcap_{j=1}^d L^2\left( \mathbb{R}^d, t_j^2\, dt \right).
$$
It is straightforward to verify that for any $t \in \mathbb{R}^d$, the operator $\exp\left((t, \nabla_{\theta})\right)$ acts on $L^2(\mathbb{R}^d)$ as
$$
\left( \exp\left((t, \nabla_{\theta})\right) \xi \right)(r) = \exp(i (t, r))\, \xi(r), \quad r \in \mathbb{R}^d,\, \xi \in L^2(\mathbb{R}^d).
$$}
Furthermore, we define the Laplace operator $\Delta_{\theta}$ as
\begin{equation}\label{laplacian}
\Delta_{\theta} = (\partial_1^{\theta})^2   + \cdots +(\partial_d^{\theta})^2, 
\end{equation}
where  $-\Delta_{\theta}$  acts as a positive operator on  $L^2(\mathbb{R}^{d}_\theta)$ (see  \cite {MSX} and \cite{Mc}).
{\color{red}Here, the operators $\Delta_{\theta}$ and $\nabla_{\theta}$ do not explicitly depend on the matrix $\theta$. However, we retain this notation to underline their association with the algebra $L^\infty(\mathbb{R}^d_\theta).$}

Through the concept of duality, we can further extend the derivatives $\partial^{\alpha}_{\theta}$ to operators on  $\mathcal{S}'(\mathbb{R}_{\theta}^d).$

In a manner of the classical setting, let us denote the pairing between $x$ from $\mathcal{S}'(\mathbb{R}_{\theta}^d)$ and  $y$  from  $\mathcal{S}(\mathbb{R}_{\theta}^d)$  as  $(x, y).$ We embed the space $L^1(\mathbb{R}^d_{\theta})+ L^\infty(\mathbb{R}^d_{\theta})$  into $\mathcal{S}'(\mathbb{R}_{\theta}^d)$ via:
$$
\langle x, y\rangle:=\tau_{\theta}(xy),\quad x\in L^1(\mathbb{R}^d_{\theta})+ L^\infty(\mathbb{R}^d_{\theta}), \quad y\in \mathcal{S}(\mathbb{R}_{\theta}^d).
$$
For a multi-index $\alpha 
\in \mathbb{N}^d_0$ and $x \in \mathcal{S}'(\mathbb{R}_{\theta}^d),$ the distribution  $\partial^\alpha_{\theta}{x}$  is defined as:
$$
\langle \partial^\alpha_{\theta}{x}, y\rangle =
(-1)^{|\alpha|}\langle x,\partial^\alpha_{\theta}{y}\rangle,\quad x\in\mathcal{S}'(\mathbb{R}_{\theta}^d), \quad y\in\mathcal{S}(\mathbb{R}_{\theta}^d). 
$$
Moreover, {\color{red} the Laplace operator $\Delta_{\theta}$ on $\mathcal{S}'(\mathbb{R}_{\theta}^d)$ is defined by} 
\begin{eqnarray}\label{extension-lapasian}
\langle \Delta_{\theta} {x}, y\rangle =\langle x, \Delta_{\theta} {y}\rangle,\quad x\in\mathcal{S}'(\mathbb{R}_{\theta}^d), \quad y\in\mathcal{S}(\mathbb{R}_{\theta}^d).   
\end{eqnarray}

\subsection{Fourier transform  on NC Euclidean spaces} 
\begin{definition}\label{F-transform}
For any $x \in \mathcal{S}(\mathbb{R}_{\theta}^d),$ we define the Fourier transform of $x$ as the map $\lambda_{\theta}^{-1}:\mathcal{S}(\mathbb{R}_{\theta}^d)\to \mathcal{S}(\mathbb{R}^d)$ by the formula
\begin{equation}\label{direct-F-transform}
\lambda_{\theta}^{-1}(x):=\widehat{x}, \quad \widehat{x}(s)=\tau_{\theta}(xU_{\theta}(s)^*), \,\ s\in \mathbb{R}^d.
\end{equation}
\end{definition}
Moreover, there is the Plancherel (Parseval) identity \cite{MSX}
\begin{equation}\label{Plancherel}
\|\lambda_{\theta}(f)\|_{L^{2}(\mathbb{R}^{d}_{\theta})}=\|f\|_{L^{2}(\mathbb{R}^{d})},\quad f\in L^{2}(\mathbb{R}^{d}).
\end{equation}

\begin{lem}\label{L-Jensen-inequality}(\cite{C1974}, \cite{D1957}, \cite[Theorem 1.20]{PTJ}) Let $H$ and $H_1$ 
be Hilbert spaces. Suppose
that $h$ is an operator convex  function on the interval $I$ and let $x\in  B(H)$ {\color{red} be a self-adjoint operator with} the spectrum $Sp(x)\subset I.$ Then we
have
\begin{eqnarray}\label{Jensen-inequality}
\pi(h(x))\geq h(\pi(x)),
\end{eqnarray}
for any positive norm\textbf{}alized linear map $\pi:B(H)\rightarrow B(H_1).$
\end{lem}
Note that, for operator concave functions, the inequality is reversed.

\subsection{Convolution on $\mathbb{R}^d_{\theta}$}

As it was already introduced in \cite[Section 3.2]{MSX}, \cite[Section 2.4]{M}), we define the convolution on  $\mathbb{R}^d_{\theta}.$

\begin{definition}\label{Def_convolution} Let $1\leq p \leq \infty$ and $x \in L^p(\mathbb{R}^d_{\theta
}).$ For $K\in L^1(\mathbb{R}^{d}),$ we define
\begin{eqnarray}\label{def-convolution}
K\ast{x}=\int\limits_{\mathbb{R}^{d}}K(\eta)T_{-\eta}(x)d\eta,
\end{eqnarray}
where the integral is understood in the sense of a $L^p(\mathbb{R}^{d})$ - valued Bochner integral when $ p < \infty,$ and as a weak* integral when $p = \infty.$ A computation using the Definition \ref{Def_convolution} shows that
\begin{equation}\label{FT_convolution}
 K\ast{x}= \lambda_\theta(\widehat{K} f).
\end{equation}
\end{definition}
The following remark can be found from \cite{Mc}.
\begin{rem}\label{positivity-convolution} The convolution \eqref{def-convolution} preserves the positivity: if $x\geq 0$ and $K \geq 0,$ then
$$ K* x\geq 0.$$
\end{rem}
\begin{definition}\label{new_conv_1}For any $x,y\in \mathcal{S}(\mathbb{R}_{\theta}^d)$ define   
$$
(x\diamond{y})(\eta):=\tau_{\theta}(xT_{-\eta}(y)), \quad \eta\in\mathbb{R}^d.
$$
\end{definition}
The following proposition shows that this is, indeed, well defined and belongs to the class of Schwartz functions.
\begin{prop}\label{traslation-function} If $x,y\in \mathcal{S}(\mathbb{R}_{\theta}^d),$ then $x\diamond{y}$ as a function belongs to $\mathcal{S}(\mathbb{R}^d).$ In other words, $x\diamond{y}\in\mathcal{S}(\mathbb{R}^d).$
\end{prop}  
\begin{proof} Let $x,y\in \mathcal{S}(\mathbb{R}_{\theta}^d).$  Then we can write $x=\lambda_{\theta}(f)$ and $y=\lambda_{\theta}(g)$ for some $f$ and $g$ in $\mathcal{S}(\mathbb{R}^d),$  respectively. Moreover,  we obtain the following equality (see, \cite[formula 2.13]{MSX})
\begin{eqnarray}\label{Notation_1}
xy=\lambda_{\theta}(f)\lambda_{\theta}(g)=\lambda_{\theta}(f*_{\theta}g),
\end{eqnarray}
where {\color{red}the $\theta$-convolution   is defined as}
$$
f*_{\theta}g(\zeta)=\int\limits_{\mathbb{R}^d}e^{\frac{i}{2}(\zeta,\theta{\xi})}f(\zeta-\xi)g(\xi)d\xi.
$$
It follows from \eqref{property-translation} and \eqref{Notation_1}   that 
\begin{eqnarray}\label{calculation_1}
xT_{-\eta}(y)\overset{\eqref{property-translation}}{=}\lambda_{\theta}(f)\lambda_{\theta}(e^{-i(\eta,\cdot)}g)\overset{\eqref{Notation_1}}{=}\lambda_{\theta}(f*_{\theta}(e^{-i(\eta,\cdot)}g)).
\end{eqnarray}  
By Definition \ref{new_conv_1} we obtain 
\begin{eqnarray}\label{calculation_2}
(x\diamond{y})(\eta)=\tau_{\theta}(xT_{-\eta}(y))\overset{\eqref{calculation_1}}{=}\int\limits_{\mathbb{R}^d}f(-\xi)e^{-i(\eta,\xi)}g(\xi)d\xi=(\widehat{\tilde{f}g})(\eta).
\end{eqnarray}
Since $f,g\in\mathcal{S}(\mathbb{R}^d),$ it follows that the function $\widehat{\tilde{f}g}$ belongs to $\mathcal{S}(\mathbb{R}^d)$, thereby completing the proof. 
\end{proof}
\begin{definition}\label{Def_convolution2} Let $K\in\mathcal{S}'(\mathbb{R}^d)$ and $x \in \mathcal{S}(\mathbb{R}_{\theta}^d).$ Then we define a convolution $K\star x$  in  $\mathcal{S}'(\mathbb{R}_{\theta}^d)$  by
\begin{eqnarray}\label{extends-convolution}
 \langle K\star x,y\rangle=\langle K, x\diamond{y}\rangle,\quad  y\in  \mathcal{S}(\mathbb{R}_{\theta}^d).  
\end{eqnarray}
\end{definition}
For example, if  $K=\delta_0$ is the Dirac mass, then 
$$
\langle \delta_0 \star x,y\rangle \overset{\eqref{extends-convolution}}{=}\langle \delta_0, x\diamond{y}\rangle=(x\diamond{y})(0)=\tau_{\theta}(xy)=\langle x,y\rangle.   
$$
Hence, 
\begin{eqnarray}\label{delta}
\delta_0 \star x=x,\quad \forall x\in\mathcal{S}(\mathbb{R}_{\theta}^d). 
\end{eqnarray} 
{\color{red}\begin{rem} In the case  the function  $K \in \mathcal{S}(\mathbb{R}^d),$ the convolution in Definition \ref{Def_convolution2} coincides with that in Definition \ref{Def_convolution}. 
\end{rem}}
The following noncommutative analogue of the Varopoulos's theorem in the noncommutative Euclidean space can be inferred from \cite{Zhao}.
\begin{thm}\cite[Theorem 3.11 and Corollary 3.19]{Zhao}\label{Zhao-thm} Let $d>2.$  Then the following inequalities are equivalent, 
\begin{eqnarray*}\label{Soblov-inequaluty}
\|x\|_{L^{\frac{2d}{d-2}}(\mathbb{R}^d_{\theta})}\leq C   \|\nabla_{\theta}{x}\|_{L^{2}(\mathbb{R}^d_\theta)}\quad \text{(Sobolev  inequality)}, 
\end{eqnarray*}
\begin{eqnarray*}\label{Nash-inequaluty}
\|x\|^{1+\frac{2}{d}}_{L^2(\mathbb{R}^d_{\theta})}\leq C  \|\nabla_{\theta}{x}\|_{L^{2}(\mathbb{R}^d_\theta)}
\|x\|^{\frac{2}{d}}_{L^1(\mathbb{R}^d_{\theta})} \quad  \quad \text{(Nash inequality)},
\end{eqnarray*}
\begin{eqnarray*}\label{Heat-kernel-estimate}
\|e^{t\Delta_{\theta}}(x)\|_{L^{\infty}(\mathbb{R}^d_{\theta})} \leq C t^{-\frac{d}{2}}   \|x\|_{L^{1}(\mathbb{R}^d_{\theta})},\quad t>0 \quad \text{(Heat kernel estimate)},
\end{eqnarray*}
with possibly different values of the constants $C>0.$
\end{thm}
\section{Hardy–Littlewood–Sobolev inequality}
In this section, we prove a quantum analogue of the Hardy–Littlewood–Sobolev inequality. First, we prove a noncommutative version of \cite[Theorem 1.2.13, p. 23]{G2008} for weak-$L^p$ spaces. 

\begin{thm}\label{(Young-inequality-weak-type-spaces}(Young inequality for weak type spaces) Let $1 \leq  p <\infty$ and
$1 < q, r <\infty $ satisfying  
$$
1+\frac{1}{r}= \frac{1}{p}+\frac{1}{q}. 
$$
Then there exists a constant $C_{p,q,r} > 0$ such that
\begin{eqnarray}\label{Young-convolution-inequality-2}
\|K*{x}\|_{L^{r,\infty}( \mathbb{R}^{d}_{\theta})} \leq  C_{p,q,r}\|K\|_{L^{q,\infty}(\mathbb{R}^{d})} \|x\|_{L^p(\mathbb{R}^{d}_{\theta})},  
\end{eqnarray} 
for all $K\in L^{q,\infty}(\mathbb{R}^{d})$ and $x\in L^p(\mathbb{R}^{d}_{\theta}).$
\end{thm}
\begin{proof} Let $M$ be a positive real number to be chosen later. Set $K_1 = K\chi_{|K|\leq M}$ and $K_2 = K\chi_{|K|>M}.$ In view of \cite[Exercise 1.1.10(a), p. 23]{G2008} we obtain 
\begin{eqnarray*}  
 d_{K_1}(\zeta)=     \left\{ \begin{array}{rcl}
         0 \;\;\;\;\;\;\;\;\;\;\;\;& \mbox{if}
         & \zeta\geq M, \\ d_{K}(\zeta)-d_{K}(M) & \mbox{if} & \zeta< M,
                \end{array}\right.    
\end{eqnarray*}
and 
\begin{eqnarray*}
 d_{K_2}(\zeta)=     \left\{ \begin{array}{rcl}
          d_{K}(\zeta)  & \mbox{if}
         & \zeta>M, \\ d_{K}(M)& \mbox{if} & \zeta\leq M.
\end{array}\right.    
\end{eqnarray*}
Here, $d_f(\cdot)$ is the the classical distribution function of $f$ (see, \cite[Definition 1.1.1. p. 2]{G2008}).
Then by  \cite[formulas (1.2.20) and (1.2.20). p. 24]{G2008}, we obtain 
\begin{eqnarray}\label{distribution-function-3}
\int\limits_{\mathbb{R}^{d}} |K_1(\xi)|^{s_1}d\xi \leq \frac{s_1}{s_1-q}M^{s_1-q}\|K\|^q_{L^{q,\infty}(\mathbb{R}^{d})} -M^{s_1}d_{K}(M), 
\end{eqnarray}
for $1\leq q<s_1<\infty,$ and 
\begin{eqnarray}\label{distribution-function-4}
\int\limits_{\mathbb{R}^{d}} |K_2(\xi)|^{s_2}d\xi  \leq \frac{q}{q-s_2}M^{s_2-q}\|K\|^q_{L^{q,\infty}(\mathbb{R}^{d})}, 
\end{eqnarray}
for $1\leq s_2<q<\infty.$

Since $1/q= 1/p'+ 1/r,$ we find $1 < q< p'.$ Now, let us choose $s_1=p'$ and $s_2=1.$
Hence, it follows from the noncommutative Young inequality \cite[Theorem 3.2, p. 16]{Mc}      and \eqref{distribution-function-3} that
\begin{eqnarray}\label{estimate-s1}
\|K_1* x\|_{L^{\infty}(\mathbb{R}^d_{\theta})}
&\leq& \|K_1\|_{L^{p'}(\mathbb{R}^{d})}\|x\|_{L^p(\mathbb{R}^{d}_{\theta})}\nonumber\\
&\overset{\eqref{distribution-function-3}}{\leq}& \left(\frac{p'}{p'-q}M^{p'-q}\|K\|^q_{L^{q,\infty}(\mathbb{R}^{d})}\right)^{\frac{1}{p'}} \|x\|_{L^p(\mathbb{R}^{d}_{\theta})}.
\end{eqnarray}
If $p' < \infty,$ then we choose $M$ such that the right-hand side of \eqref{estimate-s1} is equal to $\frac{\beta}{2},$ that is
$$ 
M = (\beta^{p'}2^{-p'}rq^{-1}\|x\|^{-p'}_{L^p(\mathbb{R}^{d}_{\theta})}
\|K\|^{-q}_{L^{q,\infty}(\mathbb{R}^{d})})^{\frac{1}{ p'-q}}
$$
and $M=\frac{\beta}{2\|x\|_{L^1(\mathbb{R}^{d}_{\theta})}}$ if $p'=\infty$ for $\beta>0.$  In this case, since  $\| K_1*_{\theta} x\|_{L^{\infty}(\mathbb{R}^d_{\theta})}\leq \frac{\beta}{2},$ we have that $d(\frac{\beta}{2};K_1*_\theta{x})= 0.$  Here, we used the fact that, if $x\in L^0(\mathbb{R}^d_{\theta})$ and  $\beta\geq0,$ then it is clear that $e^{|x|}(\frac{\beta}{2},\infty)=0 $ if and only if $x\in L^{\infty}(\mathbb{R}^d_{\theta})$ and $\|x\|_{L^{\infty}(\mathbb{R}^d_{\theta})}\leq \frac{\beta}{2}$ (see the proof of \cite[Lemma 3.2.3, p. 125]{DPS}).

Let $s_2 = 1.$ Then it follows from the noncommutative Young inequality \cite[Theorem 3.2, p. 16]{Mc} with $r=p$ and \eqref{distribution-function-4} that 
\begin{eqnarray}\label{estimate-s2}
\|K_2* {x}\|_{L^p(\mathbb{R}^d_{\theta})}
\leq \|K_2 \|_{L^1(\mathbb{R}^{d})}\|x\|_{L^p(\mathbb{R}^d_{\theta})}\ \overset{\eqref{distribution-function-4}}{=}\frac{q}{q-1}M^{1-q}\|K\|^q_{L^{q,\infty}(\mathbb{R}^{d})}\|x\|_{L^p(\mathbb{R}^d_{\theta})}.
\end{eqnarray}
Therefore, for the value of $M$ chosen, applying \eqref{estimate-s2} and the noncommutative Chebyshev inequality  \cite[Lemma 4.7, p. 4094]{D}, we have 
\begin{eqnarray*} 
d(\beta;K*{x}) &\leq& d(\frac{\beta}{2};K_2* {x}) \\
&\leq& \left(\frac{2}{\beta} \|K_2* {x}\|_{L^p(\mathbb{R}^d_{\theta})} \right)^{p}\\
&\leq&C^r_{p,q,r}\beta^{-r}\|K\|^{r}_{L^{q,\infty}(\mathbb{R}^{d})}\|x\|^{r}_{L^p(\mathbb{R}^d_{\theta})}.
\end{eqnarray*}
In other words, we obtain 
\begin{eqnarray*} 
\beta^{r}d(\beta;K*{x}) \leq C^r_{p,q,r}\|K\|^{r}_{L^{q,\infty}(\mathbb{R}^{d})}\|x\|^{r}_{L^p(\mathbb{R}^d_{\theta})}.
\end{eqnarray*}
Taking supremum over $\beta>0$ from both sides of the previous inequality we complete the proof.
\end{proof}

Next, we establish a noncommutative version of \cite[Theorem 1.4.25.  p. 73]{G2008} for weak-$L^p$ spaces. 

\begin{thm}\label{Hardy–Littlewood–Sobolev-inequality}(Hardy–Littlewood–Sobolev inequality). Let $1<p, q, r<\infty$ satisfying  
\begin{eqnarray}\label{relation-parameters}
1+\frac{1}{r}= \frac{1}{p}+\frac{1}{q}. 
\end{eqnarray}  Then, there exists $C_{p,q,r}> 0$  such that
\begin{eqnarray}\label{Hardy–Littlewood–Sobolev-inequality-1}
\|K * {x}\|_{L^r(\mathbb{R}^{d}_{\theta})} \leq C_{p,q,r} \|K \|_{L^{q,\infty}(\mathbb{R}^{d})} \|x\|_{L^p(\mathbb{R}^{d}_{\theta})}\end{eqnarray}
for $K \in L^{q,\infty}(\mathbb{R}^{d}),$ and $x\in L^p(\mathbb{R}^{d}_{\theta}).$
\end{thm}
\begin{proof}  Let $1< r,p< \infty$ and let
$p_0 < p < p_1,$   $r_0 < r < r_1,$ and $0<\gamma<1$ such that  
$$
\frac{1}{r}=\frac{1-\gamma}{r_1}+\frac{\gamma}{r_0},\quad \frac{1}{p}=\frac{1-\gamma}{p_1}+\frac{\gamma}{p_0}. 
$$
Define $A(x) := K* x.$ By Theorem \ref{(Young-inequality-weak-type-spaces}, the operator $A$ can be extended to a bounded operator from $L^{p_0}(\mathbb{R}^{d}_{\theta})$ to $L^{r_0,\infty}(\mathbb{R}^{d}_{\theta})$ and from $L^{p_1}(\mathbb{R}^{d}_{\theta})$ to $L^{r_1,\infty}(\mathbb{R}^{d}_{\theta})$ 
with constants 
$$ 
M_0=M_1=\|K\|_{L^{p,\infty}(\mathbb{R}^{d})}, 
$$ 
respectively.
Then, it follows from Corollary 7.8.3 in \cite{DPS} {\color{red} that $T$ is bounded from $L^p(\mathbb{R}^{d}_{\theta})$ to  $L^r(\mathbb{R}^{d}_{\theta}).$}
\end{proof}
{\color{red} Let $s>0.$} Throughout this paper we understand the operator $(-\Delta_{\theta})^{\frac{s}{2}}$ as a Fourier multiplier with the symbol $\sigma(\xi):=|\xi|^s,$ $\xi\in\mathbb{R}^d,$ which is defined by
\begin{equation}\label{Def_NL}
(-\Delta_{\theta})^{\frac{s}{2}} x=\lambda_{\theta}(\sigma{f}), \quad x\in \mathcal{S}(\mathbb{R}^d_{\theta}).
\end{equation}
{\color{red}  Furthermore,  for  $x\in\mathcal{S}'(\mathbb{R}_{\theta}^d)$ and  $y\in\mathcal{S}(\mathbb{R}_{\theta}^d),$ we define  the operator  $(-\Delta_{\theta})^{\frac{s}{2}}$ on $\mathcal{S}'(\mathbb{R}_{\theta}^d)$  in the following way
$$
\langle (-\Delta_{\theta})^{\frac{s}{2}}{x}, y\rangle =\langle x, (-\Delta_{\theta})^{\frac{s}{2}} {y}\rangle.   
$$}
\begin{lem}\label{laplacian-convolution} Let  $s>0.$ Then for any $K\in \mathcal{S}'(\mathbb{R}^d)$ and $x \in \mathcal{S}(\mathbb{R}_{\theta}^d)$ we have  
$$
(-\Delta_{\theta})^{\frac{s}{2}} (K \star x)=K \star (-\Delta_{\theta})^{\frac{s}{2}}x=(-\Delta)^{\frac{s}{2}}K \star x,
$$    
where $\Delta$ is the classical Laplacian. 
\end{lem}
\begin{proof} {\color{red} Let $x,y\in \mathcal{S}(\mathbb{R}_{\theta}^d).$  Then, it follows from \eqref{property-translation}, \eqref{Notation_1} and \eqref{Def_NL} that 
\begin{equation}\label{calculation_3}
(-\Delta_{\theta})^{\frac{s}{2}}x\cdot T_{-\eta}(y)\overset{\eqref{property-translation}\eqref{Def_NL}}{=}\lambda_{\theta}(\sigma{f})\lambda_{\theta}(e^{-i(\eta,\cdot)}g) 
\overset{\eqref{Notation_1}}{=}\lambda_{\theta}((\sigma{f})\ast_{\theta}(e^{-i(\eta,\cdot)}g))  \end{equation} 
and 
\begin{equation}\label{calculation_4}
x\cdot T_{-\eta}((-\Delta_{\theta})^{\frac{s}{2}}y)\overset{\eqref{property-translation}\eqref{Def_NL}}{=}\lambda_{\theta}(f)\lambda_{\theta}(e^{-i(\eta,\cdot)}(\sigma{g})) \overset{\eqref{Notation_1}}{=}\lambda_{\theta}(f\ast_{\theta}(e^{-i(\eta,\cdot)}(\sigma{g}))).
\end{equation}
By Definition \ref{new_conv_1} and using \eqref{calculation_3}-\eqref{calculation_4}, it is not hard to verify that 
\begin{equation}\label{calculation_5}
((-\Delta_{\theta})^{\frac{s}{2}}x\diamond{y})(\eta)=\tau_{\theta}((-\Delta_{\theta})^{\frac{s}{2}}xT_{-\eta}(y))  \overset{\eqref{calculation_3}}{=}  \int\limits_{\mathbb{R}^d}(\sigma\tilde{f}g)(\xi)e^{-i(\eta,\xi)}d\xi =(\widehat{\sigma\tilde{f}g})(\eta)  
\end{equation}
and 
\begin{equation}\label{calculation_6} 
(x\diamond{(-\Delta_{\theta})^{\frac{s}{2}}y})(\eta)=\tau_{\theta}(xT_{-\eta}(-\Delta_{\theta})^{\frac{s}{2}}y) \overset{\eqref{calculation_4}}{=}\int\limits_{\mathbb{R}^d}(\sigma\tilde{f}g)(\xi)e^{-i(\eta,\xi)}d\xi =(\widehat{\sigma\tilde{f}g)}(\eta),  
\end{equation}
for $\eta\in\mathbb{R}^d.$   Then, the above calculations \eqref{calculation_5} and \eqref{calculation_6}  show immediately 
\begin{equation}\label{calculation_61}
((-\Delta_{\theta})^{\frac{s}{2}}x\diamond{y})(\eta)=(x\diamond(-\Delta_{\theta})^{\frac{s}{2}}y)(\eta)\quad \text{for all}\quad x,y\in \mathcal{S}(\mathbb{R}_{\theta}^d),\eta\in\mathbb{R}^d.
\end{equation}
Thus, 
\begin{eqnarray*}  
\langle K\star (-\Delta_{\theta})^{\frac{s}{2}}x, y\rangle &\overset{\eqref{extends-convolution}}{=}&
 \langle K, (-\Delta_{\theta})^{\frac{s}{2}}x\diamond{y}\rangle\\
 &\overset{\eqref{calculation_61}}{=}& \langle K, x\diamond(-\Delta_{\theta})^{\frac{s}{2}}y\rangle \\
 &\overset{\eqref{extends-convolution}}{=}&\langle K\star x, (-\Delta_{\theta})^{\frac{s}{2}}y\rangle\\
 &\overset{\eqref{extension-lapasian}}{=}&\langle (-\Delta_{\theta})^{\frac{s}{2}}(K\star x), y\rangle\quad \text{for all} \quad y\in \mathcal{S}(\mathbb{R}_{\theta}^d).
\end{eqnarray*}
In other words, we have 
\begin{eqnarray} \label{calculation_7}
(-\Delta_{\theta})^{\frac{s}{2}} (K \star x)=K\star (-\Delta_{\theta})^{\frac{s}{2}}x ,\quad  x\in \mathcal{S}(\mathbb{R}_{\theta}^d). 
\end{eqnarray}
Let us now show that 
\begin{eqnarray}\label{calculation_8}
K \star (-\Delta_{\theta})^{\frac{s}{2}}x=(-\Delta)^{\frac{s}{2}}K \star x,\quad   x\in \mathcal{S}(\mathbb{R}_{\theta}^d). 
\end{eqnarray}
Indeed,  using \eqref{calculation_2} it is now easy to see that
\begin{eqnarray}\label{calculation_9}
(-\Delta)^{\frac{s}{2}}(x\diamond{y})(\eta)&\overset{\eqref{calculation_2}}{=}&\int\limits_{\mathbb{R}^d}\tilde{f}(\xi)\left((-\Delta)^{\frac{s}{2}}e^{-i(\eta,\xi)}\right)g(\xi)d\xi\nonumber\\
&=&\int\limits_{\mathbb{R}^d}|\xi|^{\frac{s}{2}}\tilde{f}(\xi)e^{-i(\eta,\xi)}g(\xi)d\xi\\
&=&(\widehat{\sigma\tilde{f}g})(\eta),\quad \eta\in\mathbb{R}^d.\nonumber
\end{eqnarray}
Thus, combining this calculation \eqref{calculation_9} with   \eqref{extends-convolution} and \eqref{calculation_5}, we obtain that
\begin{eqnarray*}  
\langle K\star (-\Delta_{\theta})^{\frac{s}{2}}x, y\rangle &\overset{\eqref{extends-convolution}}{=}&
 \langle K, (-\Delta_{\theta})^{\frac{s}{2}}x\diamond{y}\rangle \overset{\eqref{calculation_5}}{=}  \langle K, \widehat{\sigma\tilde{f}g}\rangle\\&\overset{\eqref{calculation_9}}{=}&\langle K,  (-\Delta)^{\frac{s}{2}}(x\diamond{y})\rangle
=\langle (-\Delta )^{\frac{s}{2}}K, x\diamond{y}\rangle\overset{\eqref{extends-convolution}}{=}\langle (-\Delta)^{\frac{s}{2}}K\star x, y\rangle. 
\end{eqnarray*}
for all $y\in \mathcal{S}(\mathbb{R}_{\theta}^d),$
which gives \eqref{calculation_8}.
Then the combination of \eqref{calculation_7} and \eqref{calculation_8} gives the desired result.}
\end{proof} 

\begin{rem}\label{Remark-convolution} Define the function on $\mathbb{R}^d$ by

\begin{eqnarray}\label{locally-integrable-function}
K_s(t):= \begin{cases}
    C_{d,s} |t|^{s-d},\quad t\in\mathbb{R}^d, \,\ \text{for}\,\ s \in \mathbb{R}_+\setminus\{d\},\\
    -C_{d}\log(|t|), \,\ \text{for}\,\ s = d,
\end{cases}
\end{eqnarray}
with constants  
$$
C_{d,s}=(2\pi)^s\frac{\pi^{-\frac{s+d}{2}}}{\pi^{s/2}}\frac{\Gamma(\frac{d+s}{2})}{\Gamma(-\frac{s}{2})} 
$$
and 
$$
C_{d}=\frac{2\pi^{d/2}}{\Gamma(d/2)},
$$
respectively, (see, \cite[Theorem 2.4.6. p. 138]{G2008}). 
In the case $s \in (0, d),$ it is well-known that this function gives the fundamental solution of the fractional Laplacian (see, \cite[Proposition 2.4.4. p. 135]{G2008}), that is,  
$$
(-\Delta)^{\frac{s}{2}}K_s=\delta_0. 
$$ 
{\color{red} The reflection $\widetilde{K_s}$ of the  distribution $K_s$ defined as follows (see, \cite[Definition 2.3.11]{G2008}) 
\begin{eqnarray}\label{property-reflection-1}
\langle \widetilde{K_s}, \phi\rangle=\langle K_s, \widetilde{\phi}\rangle,\quad \phi\in\mathcal{S}(\mathbb{R}^d). 
\end{eqnarray}
Let $x,y\in \mathcal{S}(\mathbb{R}_{\theta}^d).$ Then, it follows from  \eqref{calculation_2} and  \eqref{extends-convolution}  that
\begin{eqnarray*} 
\langle (-\Delta)^{\frac{s}{2}}K_s\star x, y\rangle\overset{\eqref{extends-convolution}}{=}\langle (-\Delta)^{\frac{s}{2}}K_s, x\diamond{y}\rangle\overset{\eqref{calculation_2}}{=}\langle (-\Delta)^{\frac{s}{2}}K_s, \widehat{\tilde{f}g}\rangle=\langle \sigma\widehat{K}_s, \tilde{f}g\rangle. 
\end{eqnarray*}
Since   $\widetilde{\widehat{K_s}}=\widehat{\widetilde{K_s}}= \widehat{K_s}$ (see, \cite[Proposition 2.3.22.]{G2008}) we have 
\begin{eqnarray*} 
 \langle (-\Delta)^{\frac{s}{2}}K_s\star x, y\rangle&=&\langle \widetilde{\sigma\widehat{K}_s}, \tilde{f}g\rangle = 
 \langle \sigma\widehat{K}_s, \widetilde{\tilde{f}g}\rangle = \langle  \sigma\widehat{K_s},  {f}\tilde{g}\rangle =\langle  \sigma\widehat{K_s}{f},  \tilde{g}\rangle\\
 &\overset{\eqref{injectivity}}{=}&(\lambda_{\theta}(\sigma\widehat{ K_s}{f}), \lambda_{\theta}(g))  \overset{\eqref{FT_convolution}}{=}(K_s* (-\Delta_{\theta})^{\frac{s}{2}}x, y). 
\end{eqnarray*}
Thus,  
\begin{eqnarray*}   
\langle \delta_0\star x, y\rangle&=
 \langle (-\Delta)^{\frac{s}{2}}K_s\star x, y\rangle    =\langle K_s* (-\Delta_{\theta})^{\frac{s}{2}} x, y\rangle 
\end{eqnarray*}
for any $y\in \mathcal{S}(\mathbb{R}_{\theta}^d).$
In other words, we obtain
\begin{equation}\label{convol-connection}
\delta_0\star x= K_s*(-\Delta_{\theta})^{\frac{s}{2}}x. 
\end{equation}}
\end{rem}

\begin{thm}\label{Sobolev-type-inequalities}(Sobolev type inequalities) Let $s > 0$ and $1 < p <  q < \infty$ satisfy 
\begin{eqnarray}\label{Parameter-condation}
\frac{1}{p}-\frac{1}{q}=\frac{s}{d}.
\end{eqnarray}
{\color{red} There exists} a constant $C_{d,p,q}>0$ such that for any $x\in\mathcal{S}(\mathbb{R}_{\theta}^d)$ we have
\begin{eqnarray}\label{Sobolev-type-inequalities-1}
 \|x\|_{L^q(\mathbb{R}^d_{\theta
})} \leq C_{d,p,q}\|(-\Delta_{\theta})^{\frac{s}{2}} x\|_{L^p(\mathbb{R}^d_{\theta
})},\quad s\in(0,d).  
\end{eqnarray}
\end{thm}
\begin{proof} Let $1<q<\infty.$  Then, applying formula \eqref{delta} and \eqref{convol-connection} in Remark \ref{Remark-convolution} we have
$$
\|x\|_{L^q(\mathbb{R}_{\theta}^d)}\overset{\eqref{delta}}{=}\|\delta_0 \star x\|_{L^q(\mathbb{R}_{\theta}^d)} \overset{\eqref{convol-connection}}{=}\| K_s* (-\Delta_{\theta})^{\frac{s}{2}}x\|_{L^q(\mathbb{R}_{\theta}^d)}. 
$$
Since $K_s \in L^{\frac{d}{d-s},\infty}(\mathbb{R}^d)$
for  $s \in (0, d)$ (see, \cite[p. 2]{VSC}) and $1 + \frac{1}{q}= \frac{d-s}{d}+\frac{1}{p}$ with $p\leq q,$ applying inequality \eqref{Hardy–Littlewood–Sobolev-inequality-1} with respect to $r=q,$ we obtain
$$
\|K_s* (-\Delta_{\theta})^{\frac{s}{2}} x\|_{L^q(\mathbb{R}_{\theta}^d)}\leq C \|(-\Delta_{\theta})^{\frac{s}{2}} x\|_{L^p(\mathbb{R}_{\theta}^d)}, 
$$
which completes the proof. 
\end{proof}


 As a particular case of inequality \eqref{Sobolev-type-inequalities-1}, when $s = 1$ and $1 < p \leq  q < \infty$ such that $\frac{1}{d}= \frac{1}{q}-\frac{1}{p},$  we  establish the following Sobolev inequality.
\begin{cor}\label{Sobolev-inequality} (Sobolev inequality). Let  $1 < p <  q < \infty$  with  $\frac{1}{d}=\frac{1}{p}-\frac{1}{q}.$ {\color{red} Then there exists} a constant $C_{d,p,q}>0$ such that for any $x\in\mathcal{S}(\mathbb{R}_{\theta}^d)$ we have
\begin{eqnarray}\label{Sobolev-inequality-1}
\|x\|_{L^q(\mathbb{R}_{\theta}^d)}\leq C_{d,p,q} \|\nabla_{\theta}x\|_{L^p(\mathbb{R}_{\theta}^d)}. 
\end{eqnarray}
\end{cor}


\section{Equivalence of the Nash and the logarithmic Sobolev inequalities} 
In this subsection, we derive the logarithmic Sobolev inequality on noncommutative Euclidean spaces for the full range \(1 < p < \infty\), removing the previous restriction \(s > 0\) present in \cite{RST}, where the result was established only for \(1 < p < 2\).
\begin{thm}(Logarithmic Sobolev inequality). \label{sobolev-ineq} Let   $1<p<\infty$ be such that $d>sp.$
Then for any $0\neq x\in\dot{W}^{1,p}(\mathbb{R}^d_\theta)$ we have the fractional
logarithmic Sobolev  inequality
$$
\tau_{\theta}\left(\frac{|x|^p}{\|x\|^p_{L^p(\mathbb{R}^d_\theta)}}\log \Big(\frac{|x|^p}{\|x\|^p_{L^p(\mathbb{R}^d_\theta)}}\Big)\right)\leq \frac{d}{p}\log \Big(C_{d,p}\frac{\|\nabla_\theta{x}\|^p_{L^p(\mathbb{R}^d_\theta)}}{\|x\|^p_{L^p(\mathbb{R}^d_\theta)}}\Big),\quad  x\in \dot{W}^{1,p}(\mathbb{R}^d_\theta),
 $$
where $C_{d,p}>0$ is a constant independent of $x$ and 
$$
\dot{W}^{1,p}(\mathbb{R}^d_\theta)=\{x\in L^p(\mathbb{R}^d_\theta): \|\nabla_\theta{x}\|_{L^p(\mathbb{R}^d_\theta)}<\infty\}.
$$
\end{thm}
\begin{proof}
The assertion follows from the Sobolev inequality in Corollary \ref{Sobolev-inequality} and  \cite[Lemma 6.3. p. 23]{RST}.   
\end{proof}
Next, we show that the Nash type inequality and the logarithmic Sobolev inequality are equivalent when $p=2.$  
\begin{thm}\label{Nash-log-Sobolev}Let $d>2.$  Then for any $0\neq x\in\mathcal{S}(\mathbb{R}_{\theta}^d)$ the following inequalities are equivalent: 
\begin{equation}\label{sobolev-log-estimate}
\tau_{\theta}\left(\frac{|x|^2}{\|x\|^2_{L^2(\mathbb{R}^d_\theta)}}\log \Big(\frac{|x|^2}{\|x\|^2_{L^2(\mathbb{R}^d_\theta)}}\Big)\right)\leq \frac{d}{2}\log \Big(C_{d,2}\frac{\|\nabla_\theta{x}\|^2_{L^2(\mathbb{R}^d_\theta)}}{\|x\|^2_{L^2(\mathbb{R}^d_\theta)}}\Big)\,\ \text{(Log-Sobolev inequality)}
 \end{equation}
and 
\begin{eqnarray}\label{Nash-inequality}
  \|x\|^{1+\frac{2}{d}}_{L^{2}(\mathbb{R}^d_\theta)} \leq C^\frac{1}{2}_{d,2}\|\nabla_\theta{x}\|_{L^{2}(\mathbb{R}^d_\theta)}\|x\|^{\frac{2}{d}}_{L^{1}(\mathbb{R}^d_\theta)}\quad \text{(Nash inequality)}. 
 \end{eqnarray}
\end{thm}
\begin{proof} Let us show that inequality \eqref{sobolev-log-estimate} implies inequality \eqref{Nash-inequality}. {\color{red} Let $h(v)=\log(\frac{1}{v}),$ $v>0,$ and let \( 0\neq x \in S(\mathbb{R}^d_\theta) \). For any \( \varepsilon > 0 \) set $|x|+\varepsilon.$ Then $|x|+\varepsilon$ is invertible and its inverse $(|x|+\varepsilon)^{-1}$ is bounded. Therefore, we have
\begin{eqnarray*}
\log\left(
\frac{\|x\|^2_{L^2(\mathbb{R}_{\theta}^d)}}{\|x\|_{L^1(\mathbb{R}_{\theta}^d)}}
\right)&=&\log\left(\frac{\||x|^2\|_{L^1(\mathbb{R}_{\theta}^d)}}{\|x\|_{L^1(\mathbb{R}_{\theta}^d)}}
\right)\\
&\leq&\log\left(
\frac{\|\,|x|(|x| + \varepsilon)\|_{L^1(\mathbb{R}_{\theta}^d)}}{\|x\|_{L^1(\mathbb{R}_{\theta}^d)}}
\right)\\
&=& h\left(
\frac{\|x\|_{L^1(\mathbb{R}_{\theta}^d)}}{\||x|(|x| + \varepsilon)\|_{L^1(\mathbb{R}_{\theta}^d)}}
\right)\\
&=& h\left(\tau_{\theta}\left(
\frac{|x|(|x| + \varepsilon)}{\||x|(|x| + \varepsilon)\|_{L^1(\mathbb{R}_{\theta}^d)}} \cdot \frac{1}{|x| + \varepsilon}
\right)\right),
\end{eqnarray*}
Notice that the map
$$
B({L^2(\mathbb{R}^d)}) \ni u \mapsto \tau_{\theta}\left(
\frac{|x|(|x| + \varepsilon)}{\||x|(|x| + \varepsilon)\|_{L^1(\mathbb{R}_{\theta}^d)}} \cdot u
\right)\in \mathbb{C}
$$
is a positive normalized linear map and $h$ is operator convex on $(0, \infty)$ (see \cite[Example 1.7]{PTJ}). Therefore, since $Sp(\frac{1}{|x|+\varepsilon})\subset (0, \infty),$ by Jensen's inequality \eqref{Jensen-inequality}, we have
\begin{eqnarray*}
h\left(\tau_{\theta}\left(
\frac{|x|(|x| + \varepsilon)}{\||x|(|x| + \varepsilon)\|_{L^1(\mathbb{R}_{\theta}^d)}} \cdot \frac{1}{|x| + \varepsilon}
\right)\right)
&\leq&\tau_{\theta}\left(
\frac{|x|(|x| + \varepsilon)}{\||x|(|x| + \varepsilon)\|_{L^1(\mathbb{R}_{\theta}^d)}} \cdot h\left(\frac{1}{|x| + \varepsilon}
\right)\right)\\
&=&
\tau_{\theta}\left(
\frac{|x|(|x| + \varepsilon)}{\||x|(|x| + \varepsilon)\|_{L^1(\mathbb{R}_{\theta}^d)}} \cdot \log(|x| + \varepsilon)
\right).
\end{eqnarray*}
Now, we need to show that
\begin{equation}\label{lim-eps}
\tau_{\theta}\left(
\frac{|x|(|x| + \varepsilon)}{\||x|(|x| + \varepsilon)\|_{L^1(\mathbb{R}_{\theta}^d)}} \cdot \log(|x| + \varepsilon)
\right)
\to
\tau_{\theta}\left(
\frac{|x|^2}{\|x\|^2_{L^2(\mathbb{R}_{\theta}^d)}} \cdot \log(|x|)\right)
\end{equation}
 as \( \varepsilon \to 0\).
Indeed, first notice
$$
|x|(|x| + \varepsilon) - |x|^2 = \varepsilon |x|,
$$
which implies that
$$
\left| \||x|(|x| + \varepsilon)\|_{L^1(\mathbb{R}_{\theta}^d)} - \|x\|^2_{L^2(\mathbb{R}_{\theta}^d)} \right| = \varepsilon \|x\|_{L^1(\mathbb{R}_{\theta}^d)} \to 0,
$$
 as \( \varepsilon \to 0\).
On the other hand, we have
$$
(|x| + \varepsilon) \log(|x| + \varepsilon) \to |x| \log(|x|)
$$
in the strong operator topology, provided that $s\mapsto s \log(s)$ is continuous on $(0, \infty)$ and is bounded on $(0,\|x\|_{L^{\infty}(\mathbb{R}_{\theta}^d)}].$ Since  $|x| \in L^1(\mathbb{R}^d_\theta),$
it follows that
$$
\tau_{\theta}(|x|(|x| + \varepsilon) \log(|x| + \varepsilon)) \to \tau_{\theta}(|x|^2 \log(|x|))\quad \text{as}\quad \varepsilon \to 0,
$$
which implies \eqref{lim-eps}.}
So, we have
\begin{eqnarray}\label{logarithmic-function-1}
\begin{split}
\log\left(
\frac{\|x\|^2_{L^2(\mathbb{R}_{\theta}^d)}}{\|x\|_{L^1(\mathbb{R}_{\theta}^d)}}
\right)\leq\tau_{\theta}\left(\frac{|x|^2}{\|x\|^2_{L^2(\mathbb{R}_{\theta}^d)}}\log(|x|)\right).
\end{split}
\end{eqnarray}
Hence, by simple properties of the logarithmic function  and inequality \eqref{sobolev-log-estimate}, we can write 
\begin{eqnarray}\label{logarithmic-function-2}
\begin{split}
\tau_{\theta}\left( \frac{|x|^2}{\||x|\|^2_{L^2(\mathbb{R}_{\theta}^d)}}\log(|x|)\right)&=\frac{1}{2}\tau_{\theta}\left( \frac{|x|^2}{\|x\|^2_{L^2(\mathbb{R}_{\theta}^d)}}\log(|x|^2)\right)\\&=\frac{1}{2}\tau_{\theta}\left( \frac{|x|^2}{\|x\|^2_{L^2(\mathbb{R}_{\theta}^d)}}\log\big(\|x\|^2_{L^2(\mathbb{R}_{\theta}^d)}\frac{|x|^2}{\|x\|^2_{L^2(\mathbb{R}_{\theta}^d)}}\big)\right)
\\&=  \log(\|x\|_{L^2(\mathbb{R}_{\theta}^d)})+\frac{1}{2}\tau_{\theta}\left( \frac{|x|^2}{\|x\|^2_{L^2(\mathbb{R}_{\theta}^d)}}\log\left(\frac{|x|^2}{\|x\|^2_{L^1(\mathbb{R}_{\theta}^d)}}\right)\right)
\\&\overset{\eqref{sobolev-log-estimate}}{\leq}   \log(\|x\|_{L^2(\mathbb{R}_{\theta}^d)})+ \frac{d}{4}\log \left(C_{d,2}\frac{\|\nabla_\theta{x}\|^2_{L^2(\mathbb{R}^d_\theta)}}{\|x\|^2_{L^2(\mathbb{R}^d_\theta)}}\right)\\
&= \log(\|x\|_{L^2(\mathbb{R}_{\theta}^d)})+  \log \left(C^{\frac{d}{4}}_{d,2}\frac{\|\nabla_\theta{x}\|^{\frac{d}{2}}_{L^2(\mathbb{R}^d_\theta)}}{\|x\|^{\frac{d}{2}}_{L^2(\mathbb{R}^d_\theta)}}\right)\\
&=\log \left(C^{\frac{d}{4}}_{2,d}\frac{\|\nabla_\theta{x}\|^{\frac{d}{2}}_{L^2(\mathbb{R}^d_\theta)}}{\|x\|^{\frac{d}{2}-1}_{L^2(\mathbb{R}^d_\theta)}}\right).
\end{split}
\end{eqnarray}
Combining \eqref{logarithmic-function-1} and \eqref{logarithmic-function-2}, we infer that
\begin{eqnarray*}  
 \frac{\|x\|^2_{L^2(\mathbb{R}_{\theta}^d)}}{\|x\|_{L^1(\mathbb{R}_{\theta}^d)}} 
\leq  C_{d,2}^\frac{d}{4}\frac{\|\nabla_\theta{x}\|^\frac{d}{2}_{L^2(\mathbb{R}^d_\theta)}}{\|x\|^{\frac{d}{2}-1}_{L^2(\mathbb{R}^d_\theta)}},  
\end{eqnarray*}
Thus, by taking both sides of the previous inequality to the power of $\frac{2}{d},$ we obtain \eqref{Nash-inequality}.

Conversely, let us consider the non-commutative H\"{o}lder inequality (see \cite[Lemma 6.3]{RST}): 
\begin{eqnarray}\label{Holder-inequality}
\|x\|_{L^r(\mathbb{R}^d_\theta)}\leq \|x\|^{\eta}_{L^p(\mathbb{R}^d_\theta)}\|x\|^{1-\eta}_{L^q(\mathbb{R}^d_\theta)}, 
\end{eqnarray}
with $\eta:=\frac{p}{r}\frac{q-r}{q-p}$ such that $p\leq r\leq q.$  Hence, by  taking the logarithm from both sides of the inequality \eqref{Holder-inequality}  we obtain
$$
\mathbf{L}(r):=\log \left(\frac{\|x\|_{L^{p}(\mathbb{R}^d_\theta)}}{\|x\|_{L^{r}(\mathbb{R}^d_\theta)}}\right)+(\eta(r)-1)\log \left(\frac{\|x\|_{L^{p}(\mathbb{R}^d_\theta)}}{\|x\|_{L^{q}(\mathbb{R}^d_\theta)}}\right)\geq0. 
$$
This inequality becomes an equality when $q = p.$  Moreover, we can rewrite the  inequality  \eqref{Holder-inequality} as follows 
\begin{eqnarray}\label{L-inequality}
\log \left(\frac{\|x\|_{L^{r}(\mathbb{R}^d_\theta)}}{\|x\|_{L^{q}(\mathbb{R}^d_\theta)}}\right)\leq  \eta \log \left(\frac{\|x\|_{L^{p}(\mathbb{R}^d_\theta)}}{\|x\|_{L^{q}(\mathbb{R}^d_\theta)}}\right).
\end{eqnarray}
For $r\leq q,$ the Logarithmic H\"{o}lder inequality  (see, \cite[Lemma 6.3]{RST}) holds
\begin{eqnarray}\label{log-H-inequality}
 \tau_{\theta}\left( \frac{|x|^r}{\|x\|^r_{L^{r}(\mathbb{R}^d_\theta)}}\log \left( \frac{|x|^r}{\|x\|^r_{L^{r}(\mathbb{R}^d_\theta)}}  \right)\right)\le   \frac{rp}{q-p} \log \left(\frac{\|x\|_{L^{q}(\mathbb{R}^d_\theta)}}{\|x\|_{L^{r}(\mathbb{R}^d_\theta)}}\right).    
\end{eqnarray}
Since 
 $$
\frac{d}{dr}\left[\frac{p}{r}\frac{q-r}{q-p}\right]=-\frac{p}{r^2}\frac{q-r}{q-p}-\frac{p}{r}\frac{1}{q-p}=-\frac{p}{r}\frac{q-r}{q-p}(\frac{1}{r}+\frac{1}{q-r})=-\frac{\eta}{r}\frac{q}{q-r}
$$ 
and 
\begin{eqnarray*}
 \frac{d}{dr}\left[ \log \left(\frac{\|x\|_{L^{p}(\mathbb{R}^d_\theta)}}{\|x\|_{L^{r}(\mathbb{R}^d_\theta)}}\right)\right]&=&-\frac{d}{dr}\left[ \frac{1}{r}\log \left( \|x\|^r_{L^{r}(\mathbb{R}^d_\theta)}  \right)\right]+\frac{d}{dr}\left[ \log \left( \|x\|_{L^{p}(\mathbb{R}^d_\theta)} \right)\right]\\
 &=&-\frac{1}{r^2}\frac{\tau_{\theta}\left(|x|^r\log|x|^r\right)}{\|x\|^r_{L^{r}(\mathbb{R}^d_\theta)}}+\frac{1}{r^2}\log \left( \|x\|^r_{L^{r}(\mathbb{R}^d_\theta)}  \right)\\
 &=&-\frac{1}{r^2}\tau_{\theta}\left(\frac{|x|^r\log|x|^r}{\|x\|^r_{L^{r}(\mathbb{R}^d_\theta)}}\right)+\frac{1}{r^2}\tau_{\theta}\left(\frac{|x|^r\log \left( \|x\|^r_{L^{r}(\mathbb{R}^d_\theta)}  \right)}{\|x\|^r_{L^{r}(\mathbb{R}^d_\theta)}}\right)\\
 &=&-\frac{1}{r^2}\tau_{\theta}\left( \frac{|x|^r}{\|x\|^r_{L^{r}(\mathbb{R}^d_\theta)}}\log \left( \frac{|x|^r}{\|x\|^r_{L^{r}(\mathbb{R}^d_\theta)}}  \right)\right),  
\end{eqnarray*}
differentiating the function $\mathbf{L}(r)$ with respect to $r$ (see \cite[Lemma 4.2]{Xiong}), and we have
\begin{eqnarray}\label{L-function-2}
\begin{split}
 \frac{d\mathbf{L}}{dr}&=-\frac{1}{r^2}\tau_{\theta}\left( \frac{|x|^r}{\|x\|^r_{L^{r}(\mathbb{R}^d_\theta)}}\log \left( \frac{|x|^r}{\|x\|^r_{L^{r}(\mathbb{R}^d_\theta)}}  \right)\right)-\frac{\eta}{r}\frac{q}{q-r} \log \left(\frac{\|x\|_{L^{p}(\mathbb{R}^d_\theta)}}{\|x\|_{L^{q}(\mathbb{R}^d_\theta)}}\right)\\
 &\overset{\eqref{log-H-inequality}}{\geq}-\frac{1}{r}\left(\frac{q}{q-r}\right) \log \left(\frac{\|x\|_{L^{q}(\mathbb{R}^d_\theta)}}{\|x\|_{L^{r}(\mathbb{R}^d_\theta)}}\right)-\frac{\eta}{r}\frac{q}{q-r} \log \left(\frac{\|x\|_{L^{p}(\mathbb{R}^d_\theta)}}{\|x\|_{L^{q}(\mathbb{R}^d_\theta)}}\right)\\
 &=-\frac{1}{r}\left(\frac{q}{q-r}\right) \left[\log \left(\frac{\|x\|_{L^{q}(\mathbb{R}^d_\theta)}}{\|x\|_{L^{r}(\mathbb{R}^d_\theta)}}\right)- \eta \log \left(\frac{\|x\|_{L^{q}(\mathbb{R}^d_\theta)}}{\|x\|_{L^{p}(\mathbb{R}^d_\theta)}}\right)\right]\\
 &\overset{\eqref{L-inequality}}{\geq}0.
 \end{split}
\end{eqnarray}
From the fact $r^2\frac{\eta}{r}\frac{q}{q-r}=\frac{pq}{q-p}$ and the estimate  \eqref{L-function-2} we get 
\begin{eqnarray*}
-r^2\frac{d\mathbf{L}}{dr}\overset{\eqref{L-function-2}}{=}\tau_{\theta}\left( \frac{|x|^r}{\|x\|^r_{L^{r}(\mathbb{R}^d_\theta)}}\log \left( \frac{|x|^r}{\|x\|^r_{L^{r}(\mathbb{R}^d_\theta)}}  \right)\right)\leq-\frac{pq}{q-p} \log \left(\frac{\|x\|_{L^{p}(\mathbb{R}^d_\theta)}}{\|x\|_{L^{q}(\mathbb{R}^d_\theta)}}\right)\\
=  \frac{p}{q-p} \log \left(\frac{\|x\|^q_{L^{q}(\mathbb{R}^d_\theta)}}{\|x\|^q_{L^{p}(\mathbb{R}^d_\theta)}}\right).
\end{eqnarray*}
In other words, 
\begin{eqnarray}\label{log-inequality}
 \tau_{\theta}\left( \frac{|x|^r}{\|x\|^r_{L^{r}(\mathbb{R}^d_\theta)}}\log \left( \frac{|x|^r}{\|x\|^r_{L^{r}(\mathbb{R}^d_\theta)}}  \right)\right)\le   \frac{p}{q-p} \log \left(\frac{\|x\|^q_{L^{q}(\mathbb{R}^d_\theta)}}{\|x\|^q_{L^{p}(\mathbb{R}^d_\theta)}}\right).    
\end{eqnarray}

If  $q=r=2$ and $p=1,$ then we can rewrite the inequality \eqref{log-inequality} as follows 
\begin{eqnarray}\label{log-inequality-1}
\tau_{\theta}\left(\frac{|x|^2}{\|x\|^2_{L^{2}(\mathbb{R}^d_\theta)}}\log\left(\frac{|x|^2}{ \|x\|^2_{L^{2}(\mathbb{R}^d_\theta)}}\right)\right)\overset{\eqref{log-inequality}}{\leq} \log \left(\frac{\|x\|^2_{L^{2}(\mathbb{R}^d_\theta)}}{\|x\|^2_{L^{1}(\mathbb{R}^d_\theta)}}\right).    
\end{eqnarray}

Next,   by the Nash type inequality \eqref{Nash-inequality}, we get 
\begin{eqnarray} \label{Nash-type-inequality} 
 \frac{\|x\|^2_{L^2(\mathbb{R}_{\theta}^d)}}{\|x\|^2_{L^1(\mathbb{R}_{\theta}^d)}} 
\leq  \left[C_{d,2}  \frac{\|\nabla_\theta{x}\|^2_{L^2(\mathbb{R}^d_\theta)}}{\|x\|^2_{L^2(\mathbb{R}^d_\theta)}}\right]^\frac{d}{2}.  \end{eqnarray}

According to \eqref{log-inequality-1} and \eqref{Nash-type-inequality},  we have 
\begin{eqnarray*}
\tau_{\theta}\left(\frac{|x|^2}{\|x\|^2_{L^{2}(\mathbb{R}^d_\theta)}}\log\left(\frac{|x|^2}{ \|x\|^2_{L^{2}(\mathbb{R}^d_\theta)}}\right)\right)\overset{\eqref{log-inequality-1}\eqref{Nash-type-inequality}}{\le} \frac{d}{2} \log \left(C_{d,2}\frac{\|\nabla_\theta{x}\|^{2}_{L^2(\mathbb{R}^d_\theta)}}{\|x\|^{2}_{L^2(\mathbb{R}^d_\theta)}}\right), 
\end{eqnarray*}
which proves \eqref{sobolev-log-estimate}. The proof is complete.
\end{proof}
As a combination of our result in Theorem \ref{Zhao-thm} with Theorem \ref{Nash-log-Sobolev} we obtain the following.
\begin{thm}\label{Main-thm} Let $d>2.$  Then for any $0\neq x\in\mathcal{S}(\mathbb{R}_{\theta}^d)$ the following inequalities are equivalent: 
\begin{eqnarray}\label{Sobolev-ineq}
\|x\|^2_{L^{\frac{2d}{d-2}}(\mathbb{R}^d_{\theta})}\leq C_{d}   \|\nabla_{\theta}{x}\|^2_{L^{2}(\mathbb{R}^d_\theta)}\quad \text{(Sobolev  inequality)}, 
\end{eqnarray}
\begin{eqnarray}\label{Nash-ineq}
\|x\|^{1+\frac{2}{d}}_{L^2(\mathbb{R}^d_{\theta})}\leq C_{d} \|\nabla_{\theta}{x}\|_{L^{2}(\mathbb{R}^d_\theta)}
\|x\|^{\frac{2}{d}}_{L^1(\mathbb{R}^d_{\theta})} \quad  \quad \text{(Nash inequality)},
\end{eqnarray}
\begin{eqnarray}\label{Heat-kernel-est}
\|e^{t\Delta_{\theta}}(x)\|_{L^{\infty}(\mathbb{R}^d_{\theta})} \leq \widetilde{C_{d}}t^{-\frac{d}{2}}   \|x\|_{L^{1}(\mathbb{R}^d_{\theta})},\quad t>0 \quad \text{(Heat kernel estimate)}, 
 \end{eqnarray}
  and 
\begin{equation}\label{sobolev-log-est}
\tau_{\theta}\left(\frac{|x|^2}{\|x\|^2_{L^2(\mathbb{R}^d_\theta)}}\log \Big(\frac{|x|^2}{\|x\|^2_{L^2(\mathbb{R}^d_\theta)}}\Big)\right)\leq \frac{d}{2}\log \Big(C_{d}^2\frac{\|\nabla_\theta{x}\|^2_{L^2(\mathbb{R}^d_\theta)}}{\|x\|^2_{L^2(\mathbb{R}^d_\theta)}}\Big)\text{(Log-Sobolev inequality)}.
\end{equation}
\end{thm}
\begin{proof}The proof is a just combination of Theorem \ref{Zhao-thm} and  Theorem \ref{Nash-log-Sobolev}.
\end{proof}

\section{Applications to Nonlinear PDEs} 

\subsection{A Gagliardo–Nirenberg type inequality} 
We now apply the Sobolev inequality to establish a Gagliardo-Nirenberg type inequality on the noncommutative Euclidean space. This inequality will then be used to study the well-posedness of certain nonlinear PDEs. The resulting Gagliardo-Nirenberg type inequality is stated below.

\begin{thm}\label{G–N-inequality} Let $d\geq 2$    and $0<{\eta}\leq1.$   Assume that 
\begin{eqnarray}\label{Parameter-condation-2}
s>0,\quad 0<r<\frac{d}{s} \quad \text{and} \quad 1\leq p\leq q \leq\frac{rd}{d-sr}. 
\end{eqnarray}
Then we have the following Gagliardo–Nirenberg type inequality
\begin{equation}\label{G–N-inequality-1}
\|x\|_{L^q(\mathbb{R}^d_{\theta})}\leq   \|(-\Delta_{\theta})^{\frac{s}{2}} x\|^{\eta}_{L^r(\mathbb{R}^d_{\theta})}\|x\|^{1-\eta}_{L^{p}(\mathbb{R}^d_{\theta})}, \quad x\in\mathcal{S}(\mathbb{R}_{\theta}^d), 
\end{equation}
where $\eta:=\frac{\frac{1}{p}-\frac{1}{q}}{\frac{s}{d}+\frac{1}{p}-\frac{1}{r}}$ and $\frac{s}{d}+\frac{1}{p}-\frac{1}{r}\ne0.$
\end{thm}
\begin{proof} {\color{red}Let $x\in\mathcal{S}(\mathbb{R}_{\theta}^d).$} Then, by the noncommutative H\"{o}lder inequality \cite[Lemma 6.3]{RST}) we find   
\begin{eqnarray}\label{H-inequality}
\|x\|_{L^q(\mathbb{R}^d_\theta)}\leq \|x\|^{\eta}_{L^{p^*}(\mathbb{R}^d_\theta)}\|x\|^{1-\eta}_{L^p(\mathbb{R}^d_\theta)},   
\end{eqnarray} 
for any $\eta \in [0, 1]$ such that   
\begin{eqnarray}\label{Parameter-condation-3}
1\leq p \leq q \leq  p^*\leq \infty\quad \text{and}\quad\frac{1}{q}=\frac{\eta}{p^*}+\frac{1-\eta}{p}.
\end{eqnarray}
On the other hand, for any  $1 <r<p^* < \infty,$ applying Theorem  \ref{Sobolev-type-inequalities} to \eqref{H-inequality} with $\frac{1}{r}-\frac{1}{p*}=\frac{s}{d}$  we obtain
$$
\|x\|_{L^q(\mathbb{R}^d_\theta)}\lesssim\|(-\Delta_{\theta})^{\frac{s}{2}} x\|^{\eta}_{L^{r}(\mathbb{R}^d_\theta)}\|x\|^{1-\eta}_{L^p(\mathbb{R}^d_\theta)}.
$$

Next, we will show that conditions \eqref{Parameter-condation-2} imply that $r<p^*$ and $\eta\in[0, 1].$  Indeed, by \eqref{Parameter-condation} and \eqref{Parameter-condation-2}, we have $\frac{1}{p^*}=\frac{1}{r}-\frac{s}{d}>0.$ Then, since $\frac{s}{d}+\frac{1}{p}-\frac{1}{r}\ne 0,$ it follows from \eqref{Parameter-condation-3} that 
$$
\eta(\frac{s}{d}+\frac{1}{p}-\frac{1}{r})=\frac{1}{p}-\frac{1}{q}\geq0,$$ 
which implies 
$$\eta=\frac{\frac{1}{p}-\frac{1}{q}}{\frac{s}{d}+\frac{1}{p}-\frac{1}{r}}.
$$
On the other hand, since
$q \leq\frac{rd}{d-sr}$ it follows from \eqref{Parameter-condation-2} that 
$$\frac{1}{p}-\frac{1}{q} \leq\frac{s}{d}+\frac{1}{p}-\frac{1}{r}, 
$$
which shows that $\eta\leq1.$
This concludes the proof.
\end{proof}
As a special case of Theorem \ref{G–N-inequality} when $p = r = 2,$ we obtain the following inequality which plays a key role for our further investigations.

\begin{cor}  Let $d\geq3$ and
$2\leq q\leq p^*=\frac{2d}{d-2}=2+\frac{4}{d-2}.$
Then we have the following inequality
\begin{equation}\label{G–N-inequality-2}
\|x\|_{L^q(\mathbb{R}^d_\theta)}\leq \|x\|^{\eta}_{L^{p^*}(\mathbb{R}^d_\theta)}\|x\|^{1-\eta}_{L^2(\mathbb{R}^d_\theta)}\leq \|(-\Delta_{\theta})^{\frac{s}{2}} x\|^{\eta}_{L^{2}(\mathbb{R}^d_\theta)}\|x\|^{1-\eta}_{L^2(\mathbb{R}^d_\theta)},\quad x\in\mathcal{S}(\mathbb{R}_{\theta}^d), 
\end{equation}  
for $\eta=\eta(q)=\frac{d(q-2)}{2d}\in[0,1].$
\end{cor}

\subsection{Linear damped wave equation on noncommutative Euclidean spaces}

For a Banach space $X,$ we denote by $C(\mathbb{R}^+; X)$ the Banach space of continuous $X$-valued  
functions on $\mathbb{R}^+=(0, \infty)$ with the norm  
$$
\| f \|_{C(\mathbb{R}^+;X)} = \sup_{s\in\mathbb{R}^+} \| f(s) \|_X.
$$
We also define the  space 
$C^1(\mathbb{R}^+; X)$ as follows
$$ 
C^1(\mathbb{R}^+; X) = \left\{ f : \mathbb{R}^+ \to X \,\middle|\, f \text{ is continuous on } \mathbb{R}^+, \text{ and } f' \text{ exists and is continuous on } \mathbb{R}^+ \right\}  
$$
with the norm 
$$ 
\|f\|_{C^1(\mathbb{R}^+; X)} = \sup\limits_{s \in \mathbb{R}^+} \|f(s)\|_X + \sup\limits_{s \in \mathbb{R}^+} \|f'(s)\|_X.
$$
For now, let us assume that  $X=L^2(\mathbb{R}^d_{\theta}).$
Given  $x \in C(\mathbb{R}^+; L^2(\mathbb{R}^d_{\theta})),$ we write   $x(t)$   for the value of  $x$  at time $t,$ and  $\partial_t x $  denotes the derivative of $x$  with respect to $t$ in the sense that  
\begin{equation}\label{Def_Par_Deri} 
\lim_{h \to 0} \left\| \frac{x(t+h) - x(t)}{h} - \partial_t x(t) \right\|_{L^2(\mathbb{R}^d_{\theta})} = 0.
\end{equation}
Moreover,  we understand the operator $\partial^{n}_t x(t)$ as follows 
$$
\lim_{h \to 0} \left\| \frac{\partial^{n-1}_tx(t+h) - \partial^{n-1}_tx(t)}{h} - \partial^{n}_t x(t) \right\|_{L^2(\mathbb{R}^d_{\theta})} = 0,\quad n \in \mathbb{N}.
$$
For each $\xi\in\mathbb{R}^d,$ it is clear that
$$
\tau_{\theta}\left(\left(\frac{x(t+h) - x(t)}{h}-\partial_{t}x(t)\right)U^*_{\theta}(\xi)\right)= \frac{\widehat{x}(t+h,\xi) - \widehat{x}(t,\xi)}{h}-\widehat{\partial_{t}x}(t,\xi). 
$$
Then, by the Plancherel  identity \eqref{Plancherel} and \eqref{Def_Par_Deri} we derive  
$$
\lim\limits_{h\rightarrow0}\left\|\frac{\widehat{x}(t+h,\xi) - \widehat{x}(t,\xi)}{h}-\widehat{\partial_{t}x}(t,\xi)\right\|_{L^2(\mathbb{R}^d)}\overset{\eqref{Def_Par_Deri}}{=}\lim\limits_{h\rightarrow0}\left\|\frac{x(t+h) - x(t)}{h}-\partial_{t}x(t)\right\|_{L^2(\mathbb{R}^d_{\theta})}=0, 
$$
which implies that $\partial_{t}\widehat{x}(t,\xi)= \widehat{\partial_{t}x}(t,\xi).$ Analogously, we can establish that
\begin{equation}\label{Partial_id}
\partial^{n}_{t}\widehat{x}(t,\xi)= \widehat{\partial^{n}_{t}x}(t,\xi),\quad \xi\in\mathbb{R}^d, n\in \mathbb{N}. 
\end{equation}
Now, we study the linear  equation
\begin{eqnarray}\label{Lin_equation}
\left\{ \begin{array}{lcl}
         \partial^2_tx(t)+(-\Delta_{\theta})x(t)+b\partial_tx(t)+mx(t)= 0,\quad t>0 ,\\ 
         x(0)=x_0\in L^2(\mathbb{R}_{\theta}^d),\\
         \partial_tx(t)=x_1\in L^2(\mathbb{R}_{\theta}^d),
                \end{array}\right.    
\end{eqnarray}
where $\Delta_{\theta}$ is the Laplace operator defined by \eqref{laplacian} and the damping is determined by $b>0$ and  $m>0.$

For any fixed each $\xi\in\mathbb{R}^d,$  applying the Fourier transform \eqref{direct-F-transform} to the equation \eqref{Lin_equation}, we have 
\begin{eqnarray}\label{ODE}
\left\{ \begin{array}{lcl}
         \partial^2_t\widehat{x}(t,\xi)+b\partial_t\widehat{x}(t,\xi)+(m+|\xi|^2)\widehat{x}(t,\xi)=0,\quad t>0,\\ 
         \widehat{x}(0,\xi)=\widehat{x}_0(\xi),\\
         \partial_t\widehat{x}(0,\xi)=\widehat{x}_1(\xi),
                \end{array}\right.    
\end{eqnarray}
where  $\widehat{x}(t,\xi):=\tau_{\theta}(x(t)U^*(\xi))$ (see, \eqref{direct-F-transform}). In the case  $|\xi|^2 + m \neq \frac{b^2}{4},$ equation \eqref{ODE} can be solved explicitly by
\begin{equation}\label{F_presentation}
\widehat{x}(t, \xi) = C_0 (\xi)e^{\left(-\frac{b}{2} + i\sqrt{|\xi|^2+m - \frac{b^2}{4}}\right)t} + C_1(\xi) e^{\left(-\frac{b}{2} - i\sqrt{|\xi|^2+m - \frac{b^2}{4}}\right)t},    
\end{equation}
where
$$
C_0 (\xi)= \left(\frac{b}{4i\sqrt{|\xi|^2+m - \frac{b^2}{4}}} +\frac{1}{2}\right)\widehat{x}_0(\xi)+\frac{1}{2i\sqrt{|\xi|^2+m - \frac{b^2}{4}}} \widehat{x}_1(\xi)
$$
and
$$
C_1 (\xi)=\left(\frac{ib}{4i\sqrt{|\xi|^2+m - \frac{b^2}{4}}} +\frac{1}{2}\right)\widehat{x}_0(\xi)+ \frac{i}{2\sqrt{|\xi|^2+m - \frac{b^2}{4}}} \widehat{x}_1(\xi).
$$
For the case $|\xi|^2 + m = \frac{b^2}{4},$ its solutions can be given by 
$$
\widehat{x}(t, \xi) =C_0e^{\frac{b}{2}t}\widehat{x}_0(\xi)+  \widehat{x}_1(\xi) t e^{-\frac{b}{2}t}, 
$$
where 
$$
C_0(\xi)=\widehat{x}_0(\xi),\quad C_1(\xi)=\frac{b}{2}\widehat{x}_0(\xi)+\widehat{x}_1(\xi).
$$
For the convenience, we can rewrite solution \eqref{F_presentation} as the following form
\begin{equation}\label{F_presentation2}
\widehat{x}(t, \xi) = \widehat{K}_0(t, \xi)\widehat{x}_0(\xi)+ \widehat{K}_1(t, \xi) \widehat{x}_1(\xi),    
\end{equation}
where the functions $\widehat{K}_0(t, \xi)$ and  $\widehat{K}_1(t, \xi)$ are expressed as
\begin{eqnarray*}
\widehat{K}_0(t,\xi) &=& \left( \frac{b}{4i\sqrt{|\xi|^2+m - \frac{b^2}{4}}} +\frac{1}{2}\right)e^{\left(-\frac{b}{2} + i\sqrt{|\xi|^2+m - \frac{b^2}{4}}\right)t}\\
&+& \left(\frac{ib}{4\sqrt{|\xi|^2+m - \frac{b^2}{4}}} +\frac{1}{2}\right)e^{\left(-\frac{b}{2} - i\sqrt{|\xi|^2+m - \frac{b^2}{4}}\right)t};
\end{eqnarray*} 
\begin{eqnarray*}
\widehat{K}_1(t,\xi) &=& \frac{1}{2i\sqrt{|\xi|^2+m - \frac{b^2}{4}}} e^{\left(-\frac{b}{2} + i\sqrt{|\xi|^2+m - \frac{b^2}{4}}\right)t}\\
&+& \frac{i}{2\sqrt{|\xi|^2+m - \frac{b^2}{4}}} e^{\left(-\frac{b}{2} - i\sqrt{|\xi|^2+m - \frac{b^2}{4}}\right)t}\quad\text{for}\quad |\xi|^2 + m \neq\frac{b^2}{4}
\end{eqnarray*} 
 and
$$
\widehat{K}_0(t,\xi) =\left(1+\frac{b}{2}t\right)e^{-\frac{b}{2}  t};\quad
\widehat{K}_1(t,\xi) =te^{-\frac{b}{2}  t}\quad \text{for}\quad |\xi|^2 + m =\frac{b^2}{4}.
$$ 
Thus, acting with the inverse Fourier transform $\lambda_{\theta}$ to \eqref{F_presentation2}, we can write the solution to equation \eqref{Lin_equation} in terms of the convolution  
\begin{equation}\label{Lin_soution}
x(t) = K_0(t) \ast x_0 + K_1(t) \ast x_1,\quad t>0, 
\end{equation}
where  $K_i(t)$ denotes the inverse Fourier transform of the functions $\widehat{K}_i(t, \xi)$ for $i = 0, 1$. Let $D(\xi):=\frac{b^2}{4}-(m+|\xi|^2).$  {\color{red} If $D(\xi)\neq 0,$ then, there exists a positive constant $c_1>0$ such that  
\begin{equation}\label{D_estimate}
c_1<|\sqrt{D(\xi)}|\leq  \frac{b}{2}\sqrt{1-\frac{4m}{b^2}}<\frac{b}{2}.
\end{equation}}
Our next step is to examine three distinct cases corresponding to signs of the discriminant $D(\xi):$

{\bf Case I}: Let $D(\xi)>0.$  For the solution of the system  \eqref{ODE} for each $\widehat{x}(t,\xi)$  we obtain the following explicit formula
$$
\widehat{x}(t,\xi) =  e^{-\frac{b}{2}t} \left(\left( \frac{b\sinh(\sqrt{D(\xi)}t)}{2\sqrt{D(\xi)}}+\cosh(\sqrt{D(\xi)}t) \right)\widehat{x}_0(\xi)+\frac{\sinh(\sqrt{D(\xi)}t)}{\sqrt{D(\xi)}}\widehat{x}_1(\xi)\right). 
$$
We use the following well-known asymptotic expansions as $\xi' \to \infty$
$$
 \sinh(\xi') = \frac{e^{\xi'} - e^{-\xi'}}{2} \sim \frac{e^{\xi'}}{2},  \quad
\cosh(\xi') = \frac{e^{\xi'} + e^{-\xi'}}{2} \sim \frac{e^{\xi'}}{2}.
$$
So, letting  $\xi':= \sqrt{D(\xi)}t,$   using \eqref{D_estimate} we get  
\begin{eqnarray*}
\frac{b\sinh(\sqrt{D(\xi)}t)}{2\sqrt{D(\xi)}t}+ \cosh(\sqrt{D(\xi)})
&\leq&\frac{b}{2\sqrt{D(\xi)}} \cdot \frac{e^{\sqrt{D(\xi)}t}}{2} + \frac{e^{\sqrt{D(\xi)}t}}{2} \\
&\overset{\eqref{D_estimate}}{\leq}& \left( \frac{b}{2c_1} + \frac{1}{2} \right)e^{\frac{b}{2}\sqrt{1-\frac{4m}{b^2}}t}
\end{eqnarray*}
and 
$$ 
\sinh(\sqrt{D(\xi)}t)\overset{\eqref{D_estimate}}{\leq} \frac{1}{2}e^{\frac{b}{2}\sqrt{1-\frac{4m}{b^2}}t} \quad \text{for}\quad t>0. 
$$
Thus, we have  the following estimate
$$
|\widehat{x}(t,\xi)|\leq C_2e^{-\frac{b}{2}\left(1-\sqrt{1-\frac{4m}{b^2}}\right)t}\left(  |\widehat{x}_0(\xi)|+\frac{1}{\sqrt{D(\xi)}}|\widehat{x}_1(\xi)|\right), 
$$
where $C_2:=\max\left\{\left( \frac{b}{2c_1}+ \frac{1}{2} \right), \frac{1}{2}\right\}$ and $\frac{b}{2}\left(1-\sqrt{1-\frac{4m}{b^2}}\right)> 0.$

{\bf Case II}: Let $D(\xi)=0.$  Then, we have
$$
\widehat{x}(t,\xi)=\left(\left(1+\frac{b}{2}t\right)\widehat{x}_0(\xi)+ \widehat{x}_1(\xi)t\right)e^{-\frac{b}{2}t}.
$$
Moreover, we obtain the estimate
$$
|\widehat{x}(t,\xi)|\leq e^{-\frac{b}{2}t}\left((1+\frac{b}{2}t)|\widehat{x}_0(\xi)|+t|\widehat{x}_1(\xi)|\right).
$$

{\bf Case III}: Let $D(\xi)<0.$ In this case, we have
$$
\widehat{x}(t,\xi)=e^{-\frac{b}{2}t}\left(\left(\cos(\sqrt{D(\xi)}t)+\frac{b\sin(\sqrt{D(\xi)}t)}{2\sqrt{D(\xi)}}\right)\widehat{x}_0(\xi)+ \frac{\sin(\sqrt{D(\xi)}t)}{\sqrt{D(\xi)}}\widehat{x}_1(\xi)\right).
$$
For $t \geq 0,$ applying inequalities  $\sin(\sqrt{|D(\xi)|}t) \leq 1$ and $\cos (\sqrt{|D(\xi)|}t) \leq 1,$ and \eqref{D_estimate}  we obtain 
$$
\left|\cos(\sqrt{D(\xi)}t)+\frac{b\sin(\sqrt{D(\xi)}t)}{2\sqrt{D(\xi)}}\right|\overset{\eqref{D_estimate}}{\leq} 1+\frac{b}{2c_1}; \quad \left|\frac{\sin(\sqrt{D(\xi)}t)}{\sqrt{D(\xi)}}\right|\leq \frac{1}{\sqrt{D(\xi)}}.
$$
Consequently, we have the following estimate 
$$
|\widehat{x}(t,\xi)|\leq C_2 e^{-\frac{b}{2}t}\left( |\widehat{x}_0(\xi)|+\frac{1}{\sqrt{D}}|\widehat{x}_1(\xi)|\right),
$$
where $C_2:=\max\left\{1+\frac{b}{2c_1}, 1\right\}.$
 
Hence,  there is a positive constant $\delta > 0$ such that in all cases, we derive
\begin{eqnarray}\label{main-estimate}
\sqrt{D(\xi)}|\widehat{x}(t,\xi)|&\leq&e^{-\delta{t}}\left(\sqrt{D(\xi)}|\widehat{x}_0(\xi)|+ |\widehat{x}_1(\xi)|\right). 
\end{eqnarray}
In this section, for any $\gamma\in\mathbb{R},$ we consider the operator $(1-\Delta_{\theta})^{\frac{\gamma}{2}}$ with the symbol $g(\xi):=(1+|\xi|^2)^\frac{\gamma}{2},$  $\xi \in \mathbb{R}^d,$ defined by
\begin{equation}\label{Def_NL2}
(1-\Delta_{\theta})^{\frac{\gamma}{2}}x= \lambda_{\theta}(gf), \quad x \in \mathcal{S}(\mathbb{R}^d_{\theta}).
\end{equation}
The Sobolev space $\mathcal{H}^\gamma(\mathbb{R}_{\theta}^d)$ with $\gamma\in\mathbb{R},$ associated to the operator \eqref{Def_NL2} is defined by 
$$
\mathcal{H}^\gamma(\mathbb{R}_{\theta}^d):=\left\{x\in\mathcal{S}'(\mathbb{R}_{\theta}^d):\quad (1-\Delta_{\theta})^{\frac{\gamma}{2}}x\in L^2(\mathbb{R}_{\theta}^d)\right\},  
$$
with the norm $\|x\|_{\mathcal{H}^\gamma(\mathbb{R}_{\theta}^d)}:=\|(1-\Delta_{\theta})^{\frac{\gamma}{2}}x\|_{L^2(\mathbb{R}_{\theta}^d)}.$ For more details on this space, we refer the reader to \cite[Subsection 3.1]{Mc}. 
\begin{prop} \label{Proposition-1} Let $s \in\mathbb{R}$ and   $x_0\in \mathcal{H}^s(\mathbb{R}_{\theta}^d)$ and 
$x_1\in \mathcal{H}^{s-1} (\mathbb{R}_{\theta}^d).$ Then there exists a positive constant $\delta > 0$ such that
\begin{eqnarray}\label{main-estimet-1}
\|x(t)\|_{\mathcal{H}^s(\mathbb{R}_{\theta}^d)}  \lesssim e^{-\delta{t}}\left(\|x_0\|_{\mathcal{H}^s(\mathbb{R}_{\theta}^d)}+\|x_1\|_{\mathcal{H}^{s-1}(\mathbb{R}_{\theta}^d)}\right),  
\end{eqnarray}
holds for all $t > 0.$ Moreover, if $x_0\in \mathcal{H}^{\alpha+s}(\mathbb{R}_{\theta}^d)$ and 
$x_1\in \mathcal{H}^{\alpha+s-1} (\mathbb{R}_{\theta}^d),$ $\alpha \in \mathbb{N}\cup\{0\},$ then we have
\begin{eqnarray}\label{main-estimet-2}
\|\partial^{\alpha}_{t}x(t)\|_{\mathcal{H}^s(\mathbb{R}_{\theta}^d)} \lesssim e^{-\delta{t}}\left(\|x_0\|_{\mathcal{H}^{\alpha+s}(\mathbb{R}_{\theta}^d)}+\|x_1\|_{\mathcal{H}^{\alpha+s-1}(\mathbb{R}_{\theta}^d)}\right), 
\end{eqnarray} 
for all $t > 0.$
\end{prop}
\begin{proof} Using the elementary asymptotic expansions   of the form  
$$
|D(\xi)|=|(m+|\xi|^2)-\frac{b^2}{4}|\sim 1+|\xi|^2\quad \text{as} \quad\xi' \to \infty,
$$ 
we can rewrite the estimate \eqref{main-estimate} as follows
\begin{equation}\label{M-estimate}
(1+|\xi|^2)^\frac{1}{2}|\widehat{x}(t,\xi)|\lesssim e^{-\delta{t}}\left((1+|\xi|^2)^\frac{1}{2}|\widehat{x}_0(\xi)|+ |\widehat{x}_1(\xi)|\right),\quad \xi\in\mathbb{R}^d. 
\end{equation}
By the Plancherel  identity \eqref{Plancherel} and the estimate \eqref{M-estimate} we obtain   
\begin{eqnarray*}
\|x(t)\|^2_{\mathcal{H}^s(\mathbb{R}_{\theta}^d)}&=&\|(1-\Delta_{\theta})^{\frac{s}{2}}x\|^2_{L^2(\mathbb{R}_{\theta}^d)} \\
&\overset{\eqref{Plancherel}}{=}&\|(1+|\xi|^2)^\frac{s-1}{2} \left((1+|\xi|^2)^\frac{1}{2}|\widehat{x}(t,\xi)|\right)\|^2_{L^2(\mathbb{R}^d)}\\ 
 &\overset{\eqref{M-estimate}}{\lesssim}&e^{-2\delta{t}}\|(1+|\xi|^2)^\frac{s}{2}  \widehat{x}_0(\xi)+(1+|\xi|^2)^\frac{s-1}{2}  \widehat{x}_1(\xi)\|^2_{L^2(\mathbb{R}^d)}\\& \lesssim&e^{-2\delta{t}}\left(\|(1+|\xi|^2)^\frac{s}{2}  \widehat{x}_0(\xi)\|^2_{L^2(\mathbb{R}^d)}+\|(1+|\xi|^2)^\frac{s-1}{2}  \widehat{x}_1(\xi)\|^2_{L^2(\mathbb{R}^d)}\right)\\
 &=&e^{-2\delta{t}}\left(\| x_0\|^2_{\mathcal{H}^{s}(\mathbb{R}_{\theta}^d)}+\| x_1\|^2_{\mathcal{H}^{s-1}(\mathbb{R}_{\theta}^d)}\right)\quad \text{for all}\quad t>0,
\end{eqnarray*}
which shows the estimate \eqref{main-estimet-1}. {\color{red} 
Hence, it follows from \eqref{F_presentation} that 
\begin{eqnarray*}
\partial^{\alpha}_{t}\widehat{x}(t,\xi)&\overset{\eqref{F_presentation}}{=}& \left(-\frac{b}{2} + i\sqrt{|\xi|^2+m - \frac{b^2}{4}}\right)^\alpha C_0 e^{\left(-\frac{b}{2} + i\sqrt{|\xi|^2+m - \frac{b^2}{4}}\right)t} \\
&+&\left(-\frac{b}{2} - i\sqrt{|\xi|^2+m - \frac{b^2}{4}}\right)^\alpha C_1 e^{\left(-\frac{b}{2} - i\sqrt{|\xi|^2+m - \frac{b^2}{4}}\right)t},    
\end{eqnarray*}
Therefore,  
\begin{equation}\label{asymptotic_ex}
|\partial^{\alpha}_{t}\widehat{x}(t,\xi)| {\leq}\left( \sqrt{|\xi|^2+m} \right)^{\alpha}|\widehat{x}(t,\xi)| \lesssim (1+|\xi|^2)^\frac{\alpha}{2}|\widehat{x}(t,\xi)|,\quad \xi\in\mathbb{R}^d.
\end{equation}}
Finally, for $\alpha \in \mathbb{N}\cup\{0\},$ by the  Plancherel  identity \eqref{Plancherel} and applying formulas \eqref{M-estimate} and \eqref{asymptotic_ex}, we have
\begin{eqnarray*}
\|\partial^{\alpha}_{t}x(t)\|^2_{\mathcal{H}^s (\mathbb{R}_{\theta}^d)}&=&\|(1-\Delta_{\theta})^{\frac{s}{2}}\left(\partial^{\alpha}_{t}x(t)\right)\|^2_{L^2(\mathbb{R}_{\theta}^d)}\\
&\overset{\eqref{Plancherel}\eqref{Partial_id}}{=}&\|(1+|\xi|^2)^{\frac{s}{2}}\partial^{\alpha}_{t}\widehat{x}(t,\xi)\|^2_{L^2(\mathbb{R}^d)} \\ 
&\overset{\eqref{asymptotic_ex}}{\lesssim}&\|(1+|\xi|^2)^{\frac{s+\alpha-1}{2}} \left((1+|\xi|^2)^{\frac{1}{2}}\widehat{x}(t,\xi)\right)\|^2_{L^2(\mathbb{R}^d)} \\
 &\overset{\eqref{M-estimate}}{\leq}&e^{-2\delta{t}}\left(\|(1+|\xi|^2)^{\frac{s+\alpha}{2}} \widehat{x}_0(\xi)\|^2_{L^2(\mathbb{R}^d)}+\|(1+|\xi|^2)^{\frac{s+\alpha-1}{2}} \widehat{x}_1(\xi)\|^2_{L^2(\mathbb{R}^d)}\right)\\
 &=&e^{-2\delta{t}}\left(\| x_0\|^2_{\mathcal{H}^{\alpha+s} (\mathbb{R}_{\theta}^d)}+\| x_1\|^2_{\mathcal{H}^{\alpha+s-1} (\mathbb{R}_{\theta}^d)}\right),\quad t>0, 
\end{eqnarray*}
thereby completing the proof.
\end{proof}

\subsection{Semilinear damped wave equations on noncommutative Euclidean spaces} 
 Let  $F:\mathbb{R}\rightarrow\mathbb{R}$ be a function satisfying the condition
\begin{eqnarray}\label{F-nonlinearity}
\left\{ \begin{array}{lcl}
         F(0)=0,\\ 
         \|F(x)-F(y)\|_{L^2(\mathbb{R}^d_{\theta})}\leq C(\|x\|^{p-1}_{L^{2p}(\mathbb{R}^d_{\theta})}+\|y\|^{p-1}_{L^{2p}(\mathbb{R}^d_{\theta})})\|x-y\|_{L^{2p}(\mathbb{R}^d_{\theta})},
                \end{array}\right.    
\end{eqnarray}
{\color{red} for some integer $p > 1$}  and self-adjoint operators $x,y\in L^{2p}(\mathbb{R}^d_{\theta}).$

The absolute value function is operator Lipschitz on all $L^p(\mathbb{R}^d_{\theta})$ spaces for $1 < p < \infty,$ for all self-adjoint operators $x,y$ such that $x-y\in L^p(\mathbb{R}^d_{\theta})$ (see \cite[Theorem 2.2]{DDdPS} and see also \cite[Corollary 3]{KPSS} for the case of normal operators). These results justify the application of the absolute value function to operators \ $x$ and $y,$ allowing us to focus solely on the case of positive operators. Therefore, the class of functions satisfying condition \eqref{F-nonlinearity} is non-empty. For instance, the function  $F(t) = t^2$ satisfies this condition for $p = 2,$ even for normal operators. Indeed, it is clear that  $F(0) = 0,$ and by the noncommutative H\"{o}lder inequality (see \cite[Theorem 1]{S2016}), we obtain the following:

$$
\|x^2 -y^2 \|_{L^{2}(\mathbb{R}^d_{\theta})}=\|x(x-y)+(x-y)y\|_{L^{2}(\mathbb{R}^d_{\theta})} \leq (\|x\|_{L^{4}(\mathbb{R}^d_{\theta})}+\|y\|_{L^{4}(\mathbb{R}^d_{\theta})}) \|x-y\|_{L^{4}(\mathbb{R}^d_{\theta})},
$$
for any normal operators $x,y\in L^4(\mathbb{R}^d_{\theta}).$   Repeating this argument inductively we can also prove that functions such as $F(t)=|t|^{p-1}t$ satisfy condition \eqref{F-nonlinearity} for positive operators and for  integer $p>1.$ Now, we consider a semilinear wave equation for the operator $-\Delta_{\theta}$ on the  noncommutative Euclidean space with the nonlinearity satisfying condition \eqref{F-nonlinearity}.

Let us study
\begin{eqnarray}\label{Semilinear-equation}
\left\{ \begin{array}{lcl}
         \partial^2_tx(t)+(-\Delta_{\theta})x(t)+b\partial_tx(t)+mx(t)= F(x(t)),\quad t>0 ,\\ 
         x(0)=x_0\in L^2(\mathbb{R}_{\theta}^d),\\
         \partial_tx(t)=x_1\in L^2(\mathbb{R}_{\theta}^d).
                \end{array}\right.    
\end{eqnarray}
with the damping term determined by $b > 0$ and with the mass $m > 0.$ 
We now formulate our main result in this subsection.
\begin{thm}\label{Theorem-global-solution} Let $F$  satisfies the properties \eqref{F-nonlinearity}. 
Assume that Cauchy data $x_0 \in \mathcal{H}^1(\mathbb{R}_{\theta}^d)$ and 
$x_1\in L^2(\mathbb{R}_{\theta}^d)$ satisfy
 \begin{eqnarray}
\|x_0\|_{\mathcal{H}^1 (\mathbb{R}_{\theta}^d)}+\|x_1\|_{L^2(\mathbb{R}_{\theta}^d)} \leq \epsilon.  
\end{eqnarray}
Then, there exists a small constant $\epsilon_0 > 0$ such that the Cauchy problem \eqref{Semilinear-equation} has a unique global solution $x\in C(\mathbb{R}^+;\mathcal{H}^1(\mathbb{R}_{\theta}^d)) \cap C^1(\mathbb{R}^+;L^2(\mathbb{R}_{\theta}^d))$ for all $0 < \epsilon \leq \epsilon_0.$ Moreover, there exists a number $\delta_0 > 0$ such that
\begin{eqnarray}\label{global-solution-estimete}
\|\partial^{\alpha}_{t}(-\Delta_{\theta})^{\beta}x(t)\|_{L^2(\mathbb{R}_{\theta}^d)} \lesssim e^{-\delta_0t}, 
\end{eqnarray}
for $(\alpha, \beta) = (0,0)$, or $(\alpha, \beta) = \left(0,\tfrac{1}{2}\right)$, or $(\alpha, \beta) = (1,0)$.

\end{thm} 
\begin{proof} Let us consider a closed subspace $\Omega_{0}$ of the space $C^1(\mathbb{R}^+;\mathcal{H}^1 (\mathbb{R}_{\theta}^d))$ defined by
$$
\Omega_{0}:=\left\{x\in C^1(\mathbb{R}^+;\mathcal{H}^1(\mathbb{R}_{\theta}^d)):\quad \|x\|_{\Omega_{0}}\leq M\right\} 
$$
with
$$
\|x\|_{\Omega_{0}}:=\sup\limits_{t\geq0}\left\{(1+t)^{-\frac{1}{2}}e^{\delta t}\left(\|x(t)\|_{L^2(\mathbb{R}_{\theta}^d)}+\|\partial_{t}x(t)\|_{L^2(\mathbb{R}_{\theta}^d)}+\|(-\Delta_{\theta})^\frac{1}{2}x(t)\|_{L^2(\mathbb{R}_{\theta}^d)}\right)\right\},
$$
where $M> 0$ will be defined later.  We define the mapping $\mathcal{A}$ on $\Omega_0$ by the formula
 \begin{equation}\label{main-mapping-1}
 (\mathcal{A}x)(t):=x_{lin}(t)+\int\limits_{0}^tK_{1}(t-s)\ast F(x(s))ds, 
\end{equation}
where {\color{red} the operator $x_{lin}(t)$ is a solution of the linear equation \eqref{Lin_equation}} and the function $K_1$ are defined in \eqref{Lin_soution}. Moreover, $\left(\mathcal{A}x\right)(t)$  is the solution of the following linear problem:
\begin{eqnarray*}
\left\{ \begin{array}{lcl}
         \partial^2_ty(t)+(-\Delta_{\theta})y(t)+b\partial_ty(t)+my(t)=0,\quad t>0, \\ 
        y(0)=0,\\
         \partial_ty(0)=F. 
\end{array}\right.
\end{eqnarray*}
We assume for a moment that 
\begin{eqnarray}\label{main-mapping-2}
\|\mathcal{A}(x)\|_{\Omega_0}\leq M,\quad \text{for all}\quad x\in \Omega_0
\end{eqnarray}
and 
\begin{eqnarray}\label{main-mapping-3}
\|\mathcal{A}(x)-\mathcal{A}(y)\|_{\Omega_0}\leq\frac{1}{r}\|x-y\|_{\Omega_0}\quad \text{for all}\quad x, y \in \Omega_0. 
\end{eqnarray}
Hence, it follows from \eqref{main-mapping-2} and \eqref{main-mapping-3} that $\mathcal{A}$ is a contraction mapping on $\Omega_0.$ The Banach fixed point theorem then implies that  $\mathcal{A}$ has a unique fixed point in $\Omega_0.$ It means that there exists a unique global solution $x$ of the equation
$$
x(t)=\mathcal{A}\left(x(t)\right)\quad \text{in} \quad  \Omega_0, 
$$
which also gives the solution to \eqref{Semilinear-equation}. Therefore, we have to prove inequalities \eqref{main-mapping-2} and \eqref{main-mapping-3}. By the Gagliardo–Nirenberg inequality \eqref{G–N-inequality-2}, and applying the following Young inequality 
$$
a^{\eta}b^{1-\eta}\leq {\eta}a+(1-\eta)b, \quad a,b\geq 0, \, 0\leq \eta \leq 1,
$$   
we obtain 
\begin{eqnarray}\label{inequality_1}  
\left(\|x\|^{p-1}_{L^{2p}(\mathbb{R}^d_{\theta})}+\|y\|^{p-1}_{L^{2p}(\mathbb{R}^d_{\theta})}\right)\|x-y\|_{L^{2p}(\mathbb{R}^d_{\theta})} 
&\overset{\eqref{G–N-inequality-2}}{\leq}& \left[\left(\|(-\Delta_{\theta})^\frac{1}{2}x\|^{\eta}_{L^{2}(\mathbb{R}^d_\theta)}\|x\|^{1-\eta}_{L^2(\mathbb{R}^d_\theta)}\right)^{p-1}\right.\nonumber\\
&+&\left.\left(\|(-\Delta_{\theta})^\frac{1}{2}y\|^{\eta}_{L^{2}(\mathbb{R}^d_\theta)}\|y\|^{1-\eta}_{L^2(\mathbb{R}^d_\theta)}\right)^{p-1} \right]\nonumber\\
&\times&\|(-\Delta_{\theta})^\frac{1}{2}(x-y)\|^{\eta}_{L^{2}(\mathbb{R}^d_\theta)}\|x-y\|^{1-\eta}_{L^2(\mathbb{R}^d_\theta)} \\
&\leq&   \left[\left(\|(-\Delta_{\theta})^\frac{1}{2}x\|_{L^{2}(\mathbb{R}_{\theta}^d)}+\|x\|_{L^{2}(\mathbb{R}_{\theta}^d)}\right)^{p-1}\right.\nonumber\\
&+&\left.\left(\|(-\Delta_{\theta})^\frac{1}{2}y\|_{L^{2}(\mathbb{R}_{\theta}^d)}+\|y\|_{L^{2}(\mathbb{R}_{\theta}^d)}\right)^{p-1}\right]\nonumber\\
&\times&\left(\|(-\Delta_{\theta})^\frac{1}{2}(x-y)\|_{L^{2}(\mathbb{R}_{\theta}^d)}+\|x-y\|_{L^{2}(\mathbb{R}_{\theta}^d)}\right). \nonumber  
\end{eqnarray}
Since $\|x\|_{\Omega_0} \leq  M$ and $\|y\|_{\Omega_0}\leq  M,$ it follows from the condition \eqref{F-nonlinearity} and \eqref{inequality_1} that
\begin{eqnarray}\label{application-G-N-1} 
\|F(x(t))-F(y(t))\|_{L^2(\mathbb{R}_{\theta}^d)}&\overset{\eqref{F-nonlinearity}}{\leq}& C\left(\|x(t)\|^{p-1}_{L^{2p}(\mathbb{R}^d_{\theta})}+\|y(t)\|^{p-1}_{L^{2p}(\mathbb{R}^d_{\theta})}\right)\|x(t)-y(t)\|_{L^{2p}(\mathbb{R}^d_{\theta})}\nonumber\\
&\overset{\eqref{inequality_1}}{\leq}&  C \left[\left(\|(-\Delta_{\theta})^\frac{1}{2}x(t)\|_{L^{2}(\mathbb{R}_{\theta}^d)}+\|x(t)\|_{L^{2}(\mathbb{R}_{\theta}^d)}\right)^{p-1}\right.\nonumber\\
&+&\left.\left(\|((-\Delta_{\theta})^\frac{1}{2})y(t)\|_{L^{2}(\mathbb{R}_{\theta}^d)}+\|y(t)\|_{L^{2}(\mathbb{R}_{\theta}^d)}\right)^{p-1}\right]\\
&\times&\left(\|(-\Delta_{\theta})^\frac{1}{2}(x(t)-y(t))\|_{L^{2}(\mathbb{R}_{\theta}^d)}+\|x(t)-y(t)\|_{L^{2}(\mathbb{R}_{\theta}^d)}\right)\nonumber\\
&\leq&C(1+t)^\frac{p}{2} e^{-\delta{t}p} \left(\|x\|^{p-1}_{\Omega_0}+\|y\|^{p-1}_{\Omega_0}\right)\|x-y\|_{\Omega_0}\nonumber\\
&\leq&C(1+t)^\frac{p}{2} e^{-\delta{t}p} 2M^{p-1} \|x-y\|_{\Omega_0}.\nonumber  
\end{eqnarray}
By putting $y(t)=0$ in \eqref{application-G-N-1}, and using that $F(0)=0,$ we also have
\begin{eqnarray}\label{application-G-N-2}
\|F(x(t))\|_{L^2(\mathbb{R}_{\theta}^d)}&\leq& 2 C(1+t)^\frac{p}{2}e^{-\delta{t}p}M^{p}.  \end{eqnarray}
Let us now proceed to estimate the integral operator as follows 
$$
(Ix)(t):=\int\limits_0^tK_{1}(t-s)\ast F(x(s))ds. 
$$ 
Let  $0\leq \beta \leq \frac{1}{2}.$ Then, from Lemma \ref{laplacian-convolution}   we derive 
$$
 ((-\Delta_{\theta})^{ \beta}Ix)(t)=\int\limits_0^tK_{1}(t-s)\ast \left((-\Delta_{\theta})^{ \beta}F(x(s))\right)ds=\int\limits_0^t\int\limits_{\mathbb{R}^d}\widehat{K}_{1}(t-s,\xi)|\xi|^{\frac{\beta}{2}}\widehat{F}(s,\xi)U_{\theta}(\xi)d\xi ds, 
$$
where $\widehat{F}(s,\xi):=\tau_{\theta}\left(F(x(s)) U_{\theta}(\xi)^*\right)$ with $\xi\in\mathbb{R}^d.$  By the Cauchy-Schwarz inequality we have
\begin{eqnarray}\label{F-integral-operator}
| \partial^{\alpha}_{t}\widehat{((-\Delta_{\theta})^{ \beta}Ix)}(t,\xi)|^2&=&\left|\partial^{\alpha}_{t}\int\limits_0^t\widehat{K}_{1}(t-s,\xi) |\xi|^{\beta}\widehat{F}(s,\xi)ds\right|^2\nonumber\\
 &\lesssim&|\widehat{K}_{1}(0,\xi)| |\xi|^{2\beta}\widehat{F}(t,\xi)|^2+\left|\int\limits_0^t\partial^{\alpha}_{t}\left(|\xi|^{\beta}\widehat{K}_{1}(t-s,\xi) \widehat{F}(s,\xi)\right)ds\right|^2\nonumber\\
 &\leq&  \frac{|\xi|^{2\beta}|\widehat{F}(t,\xi)|^2}{\sqrt{|\xi|^2+m - \frac{b^2}{4}}}+\left(\int\limits_0^t\left|\partial^{\alpha}_{t}\left(|\xi|^{\beta}\widehat{K}_{1}(t-s,\xi) \widehat{F}(s,\xi)\right)\right|ds\right)^2\\
 &\leq& (1+|\xi|^2)^{\beta-\frac{1}{2}}|\widehat{F}(t,\xi)|^2+t\int\limits_0^t\left|\partial^{\alpha}_{t}\left(|\xi|^{\beta}\widehat{K}_{1}(t-s,\xi) \widehat{F}(s,\xi)\right)\right|^2ds,\nonumber\nonumber
\end{eqnarray} 
for $\alpha= 0, 1.$

Next, it follows from  Proposition \ref{Proposition-1} and  the Plancherel (Parseval) identity \eqref{Plancherel}, \eqref{Partial_id} and \eqref{F-integral-operator} that   
\begin{eqnarray} \label{application-G-N-3} 
\|\partial^{\alpha}_{t}((-\Delta_{\theta})^{\beta}Ix)(t)\|^2_{L^2(\mathbb{R}_{\theta}^d)}&\overset{\eqref{Plancherel}\eqref{Partial_id}}{=}&  \int\limits_{\mathbb{R}^d}\left|\partial^{\alpha}_{t}\widehat{((-\Delta_{\theta})^{ \beta}Ix)}(t,\xi)\right|^2d\xi\nonumber\\
 &\overset{\eqref{F-integral-operator}}{\leq}& \int\limits_{\mathbb{R}^d} |\widehat{F}(t,\xi)|^2d\xi\nonumber\\
 &+&t\int\limits_0^t\int\limits_{\mathbb{R}^d}\left|\partial^{\alpha}_{t}\left((1+|\xi|^2)^{\frac{\beta}{2}}\widehat{K}_{1}(t-s,\xi) \widehat{F}(s,\xi)\right)\right|^2d\xi ds\\
 &\overset{\eqref{Plancherel}}{=}&\|F(x(t))\|^2_{L^2(\mathbb{R}_{\theta}^d)}+t\int\limits_0^t\left\|\partial^{\alpha}_{t} (K_{1}(t-s)\ast F(x(s)))\right\|^2_{\mathcal{H}^{\beta}(\mathbb{R}_{\theta}^d)} ds\nonumber\\
 &\overset{\eqref{main-estimet-2}}{\lesssim}&\|F(x(t))\|^2_{L^2(\mathbb{R}_{\theta}^d)}\nonumber +Ct\int\limits_0^t\left\|F(x(s))\right\|^2_{L^2(\mathbb{R}_{\theta}^d)}ds,\nonumber
\end{eqnarray}
Hence, according to  \eqref{application-G-N-1}, \eqref{application-G-N-2} and  \eqref{application-G-N-3} we find
\begin{eqnarray}\label{F-estimeite-1} 
\|\partial^{\alpha}_{t}((-\Delta_{\theta})^{\beta}(Ix-Iy))(t)\|_{L^2(\mathbb{R}_{\theta}^d)}&\overset{\eqref{application-G-N-3}}{\leq}& \left(\|F(x(t))\|^2_{L^2(\mathbb{R}_{\theta}^d)}\nonumber +Ct\int\limits_0^t\left\|F(x(s))\right\|^2_{L^2(\mathbb{R}_{\theta}^d)}ds\right)^\frac{1}{2}\nonumber\\
&\overset{\eqref{application-G-N-1}}{\leq}&(1+t)^\frac{p+1}{2} e^{-\delta{t}p} 2M^{p-1}  \|x-y\|_{\Omega_0}      
\end{eqnarray}
and
\begin{eqnarray}\label{F-estimeite-2} 
\|\partial^{\alpha}_{t}((-\Delta_{\theta})^{\beta}Ix)(t)\|_{L^2(\mathbb{R}_{\theta}^d)}&\overset{\eqref{application-G-N-2}}{\leq}& (1+t)^\frac{p+1}{2} e^{-\delta{t}p} 2M^{p}, 
\end{eqnarray}
which hold for $(\alpha,\beta)=(0,0),$ $(\alpha,\beta)=(0,\frac{1}{2}),$ and $(\alpha,\beta)=(1,0).$

Consequently, by the definition of $\mathcal{A}$ in \eqref{main-mapping-1} and using Proposition \ref{Proposition-1} for the first term and estimates for $I$ for the second term below, we obtain
\begin{eqnarray}\label{Ga-estimeite-1} 
\|\mathcal{A}\left(x\right)\|_{\Omega_0} \leq\|x_{lin}\|_{\Omega_0}+\|Ix\|_{\Omega_0} \leq C_1 
 \left(\|x_0\|^2_{\mathcal{H}^1 (\mathbb{R}_{\theta}^d)}+\|x_0\|^2_{L^1(\mathbb{R}_{\theta}^d)}\right)+C_2 M^{p}, 
\end{eqnarray}
for some $C_1 > 0$ and $C_2 > 0.$ Moreover, in the similar way, we can estimate
\begin{eqnarray}\label{Ga-estimeite-2} 
\|\mathcal{A}\left(x\right)-\mathcal{A}\left(y\right)\|_{\Omega_0}\leq \|Ix-Iy\|_{\Omega_0}\leq  C_3 M^{p-1}\|x-y\|_{\Omega_0}, 
\end{eqnarray}
for some $C_3 > 0.$ Take $r > 1$ and choose $M:= rC_1 
 \left(\|x_0\|^2_{\mathcal{H}^1 (\mathbb{R}_{\theta}^d)}+\|x_0\|^2_{L^1(\mathbb{R}_{\theta}^d)}\right)$ with sufficiently small  $\|x_0\|^2_{\mathcal{H}^1 (\mathbb{R}_{\theta}^d)}+\|x_0\|^2_{L^1(\mathbb{R}_{\theta}^d)}< \epsilon $ so that
\begin{eqnarray}\label{Constant-inequalities}
C_2M^p \leq \frac{1}{r}M,  \quad C_3M^{p-1} \leq \frac{1}{r}.
\end{eqnarray}
Hence, by using formulas \eqref{Ga-estimeite-1}–\eqref{Constant-inequalities} we obtain the estimates \eqref{main-mapping-2} and \eqref{main-mapping-3}. This allows us to apply the fixed point theorem and to establish the existence of solutions. Moreover, the estimate \eqref{global-solution-estimete} follows from \eqref{application-G-N-3}, thereby completing the proof.
 \end{proof}
{\color{red}\begin{rem}
In equation \eqref{Semilinear-equation}, the operator $-\Delta_{\theta}$ can be replaced by $(-\Delta_{\theta})^{s/2},$ but with an additional restriction for
$s.$ In this setting, the same approach remains applicable to obtain similar results to those in this section, with the corresponding version of the Sobolev space $\mathcal{H}^s(\mathbb{R}_{\theta}^d)$ used in accordance with the power of the Laplacian. 
\end{rem}}
\subsection{Application of the Nash inequality to the heat equation}

In this subsection we present a direct application of the Nash inequality to compute the decay rate for the heat equation with $-\Delta_{\theta}$ defined by \eqref{laplacian}. When $\theta=0,$ this result coincides with the application of the classical Nash's inequality which was studied by Nash himself in \cite{Nash} (see also \cite{BDS}).
\begin{thm}Let $u_{0}\in L^{1}(\mathbb{R}^d_{\theta})\cap L^{2}(\mathbb{R}^d_{\theta})$ be a positive operator and let $u(t)$ satisfies the following heat equation
\begin{equation}\label{heat-equation}
   \partial_t u(t)=\Delta_{\theta}u(t) ,\quad u(0)=u_0,\quad t>0. 
\end{equation}
Then the solution of equation \eqref{heat-equation} has the following estimate
$$
\|u(t)\|_{L^{2}(\mathbb{R}^d_{\theta})}\leq \left(\|u_0\|^{-\frac{4}{d}}_{L^{2}(\mathbb{R}^d_{\theta})}+\frac{4}{dC_{d,2}}\|u_0\|^{-\frac{4}{d}}_{L^{1}(\mathbb{R}^d_{\theta})}t\right)^{-\frac{d}{4}} 
$$
for all $t>0,$ where $C_{d,2}>0$ is the constant in \eqref{Nash-inequality}.
\end{thm}
\begin{proof} The solution of equation \eqref{heat-equation} is written as 
\begin{equation}\label{heat-solution}
u(t)=e^{-t|\cdot|^2}* u_0=\int\limits_{\mathbb{R}^d}e^{-t|\xi|^2}\widehat{u}_0(\xi)U_{\theta}(\xi)d\xi. 
\end{equation}
Since $e^{-t|\cdot|^2}>0$ and $u_0$ is a positive operator, it follows from Remark \ref{positivity-convolution} that $u(t)$ is a positive operator for any $t>0.$ Hence, by \eqref{def-convolution} and \eqref{heat-solution} we have
 $$
\|u(t)\|_{L^1(\mathbb{R}^d_{\theta})}=\tau_{\theta}(u(t))\overset{\eqref{heat-solution}}{=}\int\limits_{\mathbb{R}^d}e^{-t|\xi|^2}\widehat{u}_0(\xi)\tau_{\theta}(U_{\theta}(\xi))d\xi=\widehat{u}_0(0)=\|u_0\|_{L^1(\mathbb{R}^d_{\theta})},   
$$
which shows that 
\begin{equation}\label{solution-ineq}
\|u(t)\|_{L^1(\mathbb{R}^d_{\theta})}= \|u_0\|_{L^1(\mathbb{R}^d_{\theta})}.
\end{equation} 
 Therefore, since $u(t)$ is self-adjoint, we have 
\begin{eqnarray}\label{Soblov-estimete}
\frac{d}{dt}\|u(t)\|^2_{L^{2}(\mathbb{R}^d_{\theta})}&=&\tau_{\theta}(\partial_t u\cdot u^*)+\tau_{\theta}(u\cdot \partial_t u^*)\nonumber\\
&=&2\tau_{\theta}(\partial_t u\cdot u^*)\nonumber\\
&=&-2\tau_{\theta}(-\Delta_{\theta}u\cdot u^*)\nonumber\\
&=&-2 \|\nabla_{\theta}u\|^2_{L^{2}(\mathbb{R}^d_{\theta})}. 
\end{eqnarray}
Set $y(t):= \|u(t)\|^2_{L^{2}(\mathbb{R}^d_{\theta})}.$ Then applying the Nash inequality \eqref{Nash-inequality} with $s=1$, and by \eqref{Soblov-estimete} and \eqref{solution-ineq} we obtain 
\begin{eqnarray*}
y'\overset{\eqref{Nash-inequality}, \eqref{Soblov-estimete}}{\leq}  -2 C^{-\frac{d}{4}}_{d,2}\|u\|^{-\frac{4}{d}}_{L^{1}(\mathbb{R}^d_{\theta})}y^{1+\frac{2}{d}}\overset{\eqref{solution-ineq}}{\leq}  -2  C^{-\frac{d}{4}}_{d,2}\|u_0\|^{-\frac{4}{d}}_{L^{1}(\mathbb{R}^d_{\theta})}y^{1+\frac{2}{d}}.
\end{eqnarray*}
Integrating with respect to $t>0,$ we obtain the following estimate 
\begin{eqnarray*}
\|u(t)\|_{L^{2}(\mathbb{R}^d_{\theta})}\leq \left(\|u_0\|^{-\frac{4}{d}}_{L^{2}(\mathbb{R}^d_{\theta})}+\frac{4}{dC_{d,2}}\|u_0\|^{-\frac{4}{d}}_{L^{1}(\mathbb{R}^d_{\theta})}t\right)^{-\frac{d}{4}},
\end{eqnarray*}
which completes the proof. 
\end{proof}

\subsection{Conflict of interest}
We can conceive of no conflict of
interest in the publication of this paper. The work has not
been published previously and it has not been submitted for publication elsewhere.

\subsection{Data availability}
No new data were created or analysed during this study. Data sharing is not applicable to this article

\section{Acknowledgements}
Authors would like to thank to  Dr. Edward McDonald for helpful discussions on noncommutative Euclidean spaces. Indeed, he offered the new definition of the convolution in Definition \ref{new_conv_1}. The authors also like to thank Deyu Chen, who mentioned a gap in the proof of Theorem \ref{Nash-log-Sobolev} in the previous version. The work was partially supported by the grant from the Ministry of Science and Higher Education of the Republic of Kazakhstan (Grant No. AP23487088). 
The authors were partially supported by Odysseus and Methusalem grants (01M01021 (BOF Methusalem) and 3G0H9418 (FWO Odysseus)) from Ghent Analysis and PDE center at Ghent University. The first author was also supported by the EPSRC grant EP/V005529/1. 
Authors thank the anonymous referee for reading the paper and providing thoughtful comments, which improved the exposition of the paper.

\begin{center}

\end{center}


\begin{thebibliography}{999} 
\bibitem{BW} A.~Barchielli and R.F.~Werner. {\it Hybrid quantum-classical systems: quasi-free Markovian dynamics.} Int. J. Quantum Inf., 22(5):Pa-per No. 2440002, 51, 2024.

\bibitem{BL1976} J.~Bergh and J.~L\"{o}fstr\"{o}m. {\it Interpolation spaces. An introduction.} Springer-Verlag, Berlin-New York, 1976. Grundlehren der Mathematischen Wissenschaften, no. 223. 

\bibitem{BDS}E. Bouin, J. Dolbeault, Ch. Schmeiser, {\it A variational proof of Nash’s inequality}. Atti Accad. Naz. Lincei Cl. Sci. Fis. Mat. Natur. 31 (2020), no. 1, 211--223.

\bibitem{CGRS} A.L.~Carey, V.~Gayral, A.~Rennie  and F.~Sukochev. {\it Index theory for locally compact noncommutative geometries.} Mem. Amer. Math. Soc., 231(2014), no. 1085, vi+130. 

\bibitem{C1974} M.D.~Choi.  {\it A Schwarz inequality for positive linear maps on C*-algebras.} Illinois J. Math. 18(1974), 565-574.

 \bibitem{DW} L.~Dammeier and R.F.~Werner. {\it Quantum-classical hybrid systems and their quasifree transformations.} Quantum, 7(2023), 1068pp.

\bibitem{D1957} C.~Davis.  {\it A Schwartz inequality for convex operator functions.} Proc. Amer. Math. Soc. 8(1957), 42–44.

\bibitem{DDdPS} P.G.~Dodds,  T.K.~Dodds,  B.~de Pagter and   F.A.~Sukochev. {\it Lipschitz continuity of the absolute value and Riesz projections in symmetric operator spaces.} J. Funct. Anal. 148(1997), no. 1, 28-69.

\bibitem{DPS} P.G.~Dodds, B.~de Pagter and F.A.~Sukochev. {\it Noncommutative Integration and Operator Theory.} Springer, Berlin, Heidelberg. 2023. 

\bibitem{FK} T.~Fack and H.~Kosaki.  {\it Generalized $s$-numbers of $\tau$-measurable operators.} Pacific J. Math.  2 (1986), no. 123, 269-300.

\bibitem{GJM} L.~Gao, M.~Junge and E.~McDonald. {\it Quantum Euclidean spaces with noncommutative derivatives.}  J. Noncommut. Geom. 16(2022), 153–213.

\bibitem{GGSVM} V.~Gayral, J.M.~Gracia-Bondia, B.~Iochum, T.~Sch\"ucker and J.C.~Varilly. {\it Moyal planes are spectral triples.} Comm. Math. Phys. 246(2004), no. 3, 569-623.

\bibitem{GT} D.~Gilbarg and N.~Trudinger. {\it Elliptic partial differential equations of second order.}  (Sec. edit.), Springer-Verlag, 1983.

\bibitem{GJP} A.M.~Gonz\'{a}lez-P\'{e}rez, M.~Junge  and J.~Parcet. {\it Singular integrals in quantum Euclidean spaces.}  Mem. Amer. Math. Soc. 272 (2021), no. 1334, xiii+90 pp.

\bibitem{Gross} L.~Gross. {\it Logarithmic sobolev inequalities.} Amer. J. Math. 97(1975), no. 4, 1061–1083.

\bibitem{Gross2} L.~Gross. {\it Hypercontractivity and logarithmic Sobolev inequalities for the Clifford Dirichlet form.} Duke Math. J. 42(1975), no. 3, 383-396.
 
\bibitem{green-book} J.M.~Gracia-Bondia, J.~V\'{a}rilly and H.~Figueroa.  {\it  Elements of noncommutative geometry.}  Birkh\"{a}user Boston, Inc., Boston, MA, 2001. xviii+685 pp.

\bibitem{Gro} H.J.~Groenewold. {\it On the principles of elementary quantum mechanics.} Physica. 12(1946), no. 7, 405-460.

\bibitem{G2008} L.~Grafakos. {\it Classical Fourier Analysis}. Second Edition. Springer-Verlag, New York (2008) 

\bibitem{H} B.C.~Hall.  {\it Quantum theory for mathematicians.} volume 267 of Graduate Texts in Mathematics. Springer,
New York, 2013.

\bibitem{HKN04} N.~Hayashi, E.I.~Kaikina, and P.I.~Naumkin. Damped wave equation with super critical nonlinearities. Differential Integral Equations, 17 (2004), 637–652.

\bibitem{HKN06} N.~Hayashi, E.I.~Kaikina, and P.I.~Naumkin. Damped wave equation with a critical nonlinearity. Trans. Amer. Math. Soc., 358 (2006), 1165–1185.

\bibitem{HLW} G.~Hong, X.~Lai and  L.~Wang. {\it Fourier restriction estimates on quantum Euclidean spaces.} Adv. Math. 430(2023), 109232, 24 pp.

\bibitem{HL92} L.~Hsiao, T.-P.~Liu. Convergence to nonlinear diffusion waves for solutions of a system of hyperbolic conservation laws with damping. Comm. Math. Phys., 43 (1992), 599–605.

\bibitem{HO04} T.~Hosono, T.~Ogawa. Large time behavior and Lp–Lq estimate of 2-dimensional nonlinear damped wave equations. J. Differential Equations, 203 (2004), 82–118.

\bibitem{Ike04} R.~Ikehata. New decay estimates for linear damped wave equations and its application to nonlinear problem. Math. Meth. Appl. Sci., 27 (2004), 865–889.

\bibitem{KU13} T.~Kawakami and Y.~Ueda. Asymptotic profiles to the solutions for a nonlinear damped wave
equation. Differential Integral Equations, 26 (2013), 781–814.


\bibitem{KPSS} \^{E}.~Kissin, D.S.~Potapov,  F.A.~Sukochev and V.S.~Shulman, {\it Lipschitz functions, Schatten ideals, and unbounded derivations.(Russian) Funktsional.} Anal. i Prilozhen.45(2011), no. 2, 93–96; translation in Funct. Anal. Appl. 45(2011), no. 2, 157–159.


\bibitem{Kha13} M.~Khader. Global existence for the dissipative wave equations with spacetime dependent potential. Nonlinear Anal., 81 (2013), 87–100.


\bibitem{L} L.~Lafleche. {\it On quantum Sobolev inequalities.} J. Funct. Anal. 286(2024), no. 10, 110400.

\bibitem{LSZ} S.~Lord, F.~Sukochev and  D.~Zanin. {\it Singular traces. Theory and applications.} De Gruyter Studies in Mathematics,  46. De Gruyter, Berlin, 2013.

\bibitem{Matsumura} A.~Matsumura. { \it On the asymptotic behavior of solutions of semi-linear wave equations.} Publ. Res. Inst. Math. Sci., 12 (1976/77), no. 1, 169-189.

\bibitem{MSX}E.~McDonald, F.~Sukochev and X.~Xiong. {\it Quantum differentiability on noncommutative Euclidean spaces.} 
Comm. Math. Phys. 379(2020), no. 2, 491-542.

\bibitem{Mc}E.~McDonald. {\it Nonlinear partial differential equations on noncommutative Euclidean spaces.} J. Evol. Equ. 24(2024), 16pp. 

\bibitem{Mer} J.~Merker. {\it Generalizations of logarithmic Sobolev inequalities. Discrete Contin.} Dyn. Syst. Ser. S. 1(2008), no. 2, 329–338.

\bibitem{Moser} J. Moser. {\it On Harnack's theorem for elliptic differential equations.} Comm. Pure Appl. Math. 14 (1961), 577-591.

\bibitem{M} J.E.~Moyal. {\it Quantum mechanics as a statistical theory.} Proc. Cambridge Philos. Soc. 45(1949), 99-124.

\bibitem{Nash}J. Nash. {\it Continuity of solutions of parabolic and elliptic equations.} Amer. J. Math. 80(1958), 931-954. 

\bibitem{PXu} G.~Pisier and Q.~Xu. {\it Non-commutative $L_p$-spaces.  In Handbook of the geometry of Banach spaces.} North-Holland, Amsterdam. 2(2003),  1459–1517.
 
\bibitem{PTJ} J.~P\'{e}cari\'{c},  T.~Furuta,  Hot J.~Mi\'{c}i\'{c} and Y.~Seo. {\it Mond-P\'{e}cari\'{c} method in operator inequalities: inequalities for bounded selfadjoint operators on a Hilbert space.} Monographs in Inequalities. Element, Zagreb, 2005.


\bibitem{Rieffel} M.~Rieffel.  {\it Deformation quantization for actions of $\mathbb{R}^d.$} Mem. Amer. Math. Soc. 106(1993), no 506, 93pp.

\bibitem{RV} J.~Rozendaal and  M.~Veraar. {\it Fourier multiplier theorems involving type and cotype.} J. Fourier Anal. Appl. 24(2018), no. 2, 583–619.

\bibitem{RST} M.~Ruzhansky, S.~Shaimardan and  K.~Tulenov. {\it $L^p-L^q$ boundedness of Fourier multipliers on quantum Euclidean spaces}. 	arXiv:2312.00657

\bibitem{RST2} M.~Ruzhansky, S.~Shaimardan and  K.~Tulenov. {\it Fourier multipliers and their applications to PDE on the quantum Euclidean space}. NoDEA, (2025). 1-21. In Press. arXiv:2504.07604 

\bibitem{RT19} M.~Ruzhansky and N.~Tokmagambetov. {\it On nonlinear damped wave equations for positive operators. I. Discrete spectrum.} Differential Integral Equations 32(2019), no 7/8, 455-478. 

\bibitem{RT} M.~Ruzhansky and N.~Tokmagambetov. {\it Nonlinear damped wave equations for the sub-Laplacian on the Heisenberg group and for Rockland operators on graded Lie groups}. J. Differ. Equ. 265 (2018), 5212-5236.

\bibitem{D} D.~Sjoerd. {\it Noncommutative Boyd interpolation theorems.}
Trans. Amer. Math. Soc. 367 (2015), no. 6, 4079-4110.

\bibitem{Sobolev} S.~Sobolev. {\it On a theorem of functional analysis.} Amer. Math. Soc. Translations 34(1963), 39-68.

\bibitem{Stein} E.M.~Stein. {\it Singular integrals and differentiability properties of functions.} Princeton Mathematical Series, No. 30. Princeton University Press, Princeton, N.J., 1970.


\bibitem{S2016} F.~Sukochev, {\it H\"older inequality for symmetric operator spaces and trace property of $K$-cycles.} Bull. London Math. Soc. 48 (2016) 637–647.

\bibitem{Var} N.~Varopoulos. {\it Hardy-Littlewood theory for semigroups}. J. Funct. Anal. 63(1985), no. 2, 240-260.

\bibitem{VSC} N.~Varopoulos, L.~Saloff-Coste and  T.~Coulhon. {\it Analysis and geometry on groups.} Cambridge University Press, 1993.


\bibitem{von} W. von Wahl. \"{U}ber die klassische Losbarkeit des Cauchy-Problems fur nichtlineare Wellen- gleichungen bei kleinen Anfangswerten und das asymptotische Verhalten der Losungen. (Ger- man) Math. Z., 114 (1970), 281–299.

\bibitem{W} R.~Werner. {\it Quantum harmonic analysis on phase space.} J. Math. Phys. 25(1984), no. 5, 1404–1411.


\bibitem{Xiong} X.~Xiong. {\it Noncommutative harmonic analysis on semigroup and ultracontractivity.} Indiana University Mathematics Journal. 66(2017), no. 6, 1921-1947.



 \bibitem{Zhao} M.~Zhao. {\it Smoothing estimates for non commutative spaces.} PhD thesis, University of Illinois at Urbana-Champaign, 2018. http://hdl.handle.net/2142/101570.

\end{thebibliography}
\end{document}